\documentclass{amsart}
\usepackage{graphicx}
\usepackage{amsmath}
\usepackage{amssymb}
\usepackage{amsthm}
\usepackage{amsfonts}
\usepackage{hyperref}
 \usepackage{tikz}
\usetikzlibrary{arrows.meta,decorations.pathreplacing,calc,positioning,fit,
	shapes.geometric,backgrounds,decorations.markings}
\tikzset{
	bp/.style={circle,fill=black,inner sep=1.3pt},
	rt/.style={circle,fill=white,draw=black,inner sep=1.4pt},
	edge1/.style={line width=1.3pt},
	edge2/.style={line width=0.95pt},
	edge3/.style={line width=0.65pt},
	edged/.style={line width=0.5pt,densely dotted},
	lab/.style={font=\scriptsize},
	brc/.style={decorate,decoration={brace,amplitude=5pt},line width=0.6pt},
}	

\DeclareMathOperator{\End}{\mathrm{End}}
\DeclareMathOperator{\diam}{\mathrm{diam}}
\DeclareMathOperator{\dist}{\mathrm{dist}}
\DeclareMathOperator{\mesh}{\mathrm{mesh}}
\DeclareMathOperator{\Int}{\mathrm{Int}}
\DeclareMathOperator{\G}{\mathcal{G}}
\DeclareMathOperator{\Reg}{\mathrm{Reg}}
\newcommand{\Lev}{\mathcal{L}}

\newcommand{\htop}{h_{\textnormal{top}}}

\newtheorem{theorem}{Theorem}[section]
\newtheorem{lemma}[theorem]{Lemma}
\newtheorem{corollary}[theorem]{Corollary}
\newtheorem{question}[theorem]{Question}
\newtheorem{step}{Step}

\newtheorem{claim}{Claim}

\theoremstyle{definition}
\newtheorem{definition}[theorem]{Definition}

\theoremstyle{remark}
\newtheorem{remark}[theorem]{Remark}
\newtheorem*{remarknn}{Remark}

\title[MMEs for Markovian dynamics on the Gehman dendrite]{Measures of maximal entropy for Markovian dynamics on the Gehman dendrite}
\author{Piotr Oprocha} 
\address[P. Oprocha]{
	National Supercomputing Centre IT4Innovations, University of Ostrava,
	IRAFM,
	30. dubna 22, 70103 Ostrava,
	Czech Republic
    }
\email{piotr.oprocha@osu.cz}


\author{Jakub Tomaszewski}
\address[J. Tomaszewski]{
	AGH University of Krakow, Faculty of Applied Mathematics,
    al.\ Mickiewicza 30,
    30-059 Krak\'ow, Poland.
    -- $\&$ --
    University of Maryland,
    Department of Mathematics,
    William E. Kirwan Hall,
    4176 Campus Dr., College Park,
    MD 20742 USA.
}
\email{tomaszew@agh.edu.pl}

\begin{document}
\begin{abstract}
    We study transitive dynamical systems on the Gehman dendrite $\mathcal{G}$ for which the endpoint set $\mathrm{End}(\mathcal{G})$ is invariant. Our goal is to approximate such systems by maps whose measure-theoretic behaviour at maximal entropy is governed by an explicit countable Markov structure. We introduce a class of Markovian maps, encode their dynamics by countable Markov graphs, and use the criteria of Vere-Jones, Gurevich, Salama and Ruette to control the existence of measures of maximal entropy. The main theorem gives two arbitrarily close mixing Markovian perturbations of any given system in the considered class: one has a unique measure of maximal entropy, while the other has none.
\end{abstract}

\maketitle

\section{Introduction}
The interaction between topological and measure-theoretic complexity is a central theme in modern topological dynamics and ergodic theory. One of its basic manifestations is the variational principle associated with topological entropy. For a dynamical system $(X,f)$, the topological entropy $h_{top}(X,f)$ is the supremum of the measure-theoretic entropies $h_\mu(X,f)$ over all $f$-invariant probability measures $\mu$. An ergodic measure attaining this supremum is called a measure of maximal entropy (MME). Thus, the existence and uniqueness of MMEs record how the topological complexity of a system is represented at the level of invariant measures.

For smooth systems this question is subject to strong compactness and regularity mechanisms. Newhouse proved that smooth diffeomorphisms on compact manifolds without boundary admit MMEs \cite{Newhouse1989}. A subsequent result by Buzzi, Crovisier and Sarig shows that a smooth diffeomorphism of a closed surface with positive topological entropy has only finitely many MMEs, and that the MME is unique in the transitive case \cite{Buzzi2022}.

In their recent work, Buzzi, Crovisier, and Sarig introduced the concept of strong positive recurrence (SPR) for surface diffeomorphisms \cite{Buzzi2025}. They showed that all smooth surface diffeomorphisms with positive entropy are SPR. Thus, the SPR property is common, but also very powerful. It follows that all SPR dynamical systems can be encoded by countable-state topological Markov chains, whose transition matrices act with a spectral gap - and this property has been used to show a variety of stochastic properties, ranging from exponential decay of correlations (and hence exponential mixing) to almost sure invariance principles and large deviations for the ergodic measures (see also \cite{Gouzel2010, Hardy1988, Kifer1990, Pollicott1990, Ruelle2004}).

These previously mentioned countable topological Markov chains have been of core interest in modern theory of dynamical systems. Their utility in encoding more complex systems has proven extremely useful not only in \cite{Buzzi2022, Buzzi2025}, but also in e.g. \cite{Encode1,Encode2,Encode3,Encode4,Encode5,Encode6,Encode7}. Recent work has also extensively explored the thermodynamic formalism for topological Markov chains \cite{Thermo1,Thermo2,Thermo3,Thermo4,Thermo7,Thermo5,Thermo6}, as well as the relations between and characterizations of topological Markov chains \cite{Coding1,Coding3,Coding2,Coding4,Coding3.5}. A landmark achievement in this area is the complete classification of strongly positive recurrent topological Markov chains by Boyle, Buzzi, and Gomez \cite{Boyle2006}. Their work provides a full classification, up to entropy conjugacy, of the natural extensions of smooth entropy-expanding maps.

When departing from smooth systems on manifolds towards low dimensional branched spaces, the structural behaviour of MMEs can exhibit strikingly different properties, depending on the characteristics of the underlying system. Hofbauer showed that finitely piecewise monotone maps mirror the behaviour of smooth surface diffeomorphisms: positive entropy guarantees finite positive number of MMEs and the MME is unique in the transitive case \cite{Encode2, Encode3}. Crucially, these maps can be encoded by a finite topological Markov chain and thus Hofbauer's results are in line with Parry's findings \cite{Parry}. Moreover, under strong rigidity conditions the same assertions apply even when encoding by countably infinite topological Markov chains. In \cite{Encode1}, Buzzi showed that Hofbauer's results can be extended to smooth interval maps, for which finite topological Markov chains may not be sufficient. However, weakening these assumptions to $C^r$ interval maps yields transitive \cite{Encode1} or mixing \cite{Ruette_natural_extension} systems without an MME.

These results motivate the following comparison problem: to what extent can analogous conclusions survive in nontrivial one dimensional continua where finite Markov partitions are no longer available? Moving to more complex one dimensional structures like dendrites — compact, connected, locally connected metric spaces containing no simple closed curves — presents fundamentally new challenges. Dendrites serve as a natural laboratory bridging one dimensional and higher dimensional dynamics; they are complex enough to mirror rich symbolic behaviour while remaining low dimensional enough to be analytically tractable.

Recent work has clarified the possible entropy values for dynamics on dendrites. For instance, Kwietniak and the authors resolved the ``entropy paradox'' on the Gehman dendrite $\mathcal{G}$ by proving that the infimum of the topological entropy for pure mixing maps is zero \cite{Entropy_paradox_dendryty}, thus in agreement with the behaviour of exact maps \cite{Spitalsky2013}. These constructions allow for arbitrary positive values of the topological entropy and hence provide a full attainable range, as the topological entropy of transitive maps on $\G$ must be positive \cite{Positive_entropy_Gehman}. However, much less is known about the measures that realize these values. Thus, these previously mentioned results leave open the finer question addressed here: 

\begin{question}\label{que:main}
Can one force, or destroy, measures of maximal entropy by arbitrarily small perturbations?     
\end{question}

Since the Gehman dendrite has countably many branch points and a Cantor set of endpoints, finite Markov partition methods are inadequate for such a nontrivial space. The natural substitute is a countable symbolic representation and the structural graph criteria for countable topological Markov chains developed by Vere-Jones \cite{Vere}, Gurevich \cite{Gurevich_infinite_graphs_no_MME, Gurevich_MME_iff}, Salama \cite{Salama_entropy_R} and Ruette \cite{Ruette}. It is worth noting that similar techniques have been used to study tame graphs, yet another class of nontrivial one dimensional continua \cite{BARTO2019}.

We address Question~\ref{que:main} for transitive maps $f\colon \G\to\G$ satisfying $f(\End(\G))=\End(\G)$. The endpoint invariance assumption keeps the Cantor subsystem separated from the interior countably infinite coding and allows us to compare the entropy of the full system with the entropy of the induced countable graph. The approximating systems belong to a structured class that we call Markovian dynamics. Informally, a map is $P$-Markovian if it is piecewise linear on the free arcs determined by a countable invariant set $P\subset \G\setminus \End(\G)$ and if its transitions are recorded by a countable directed graph. This gives enough flexibility to approximate arbitrary systems in the class, but enough symbolic structure to decide whether MMEs exist.

The main result says that this class is dense in two opposite ways. For every transitive system $(\G,f)$ with invariant endpoint set and every $\gamma>0$, there are two mixing Markovian maps $\hat f_{\textnormal{NME}}$ and $\hat f_{\textnormal{MME}}$ within $\gamma$ of $f$ in the supremum metric, but with opposite behaviour at maximal entropy:
\begin{enumerate}
    \item The system $(\G, \hat{f}_{\text{NME}})$ is encoded by a countable topological Markov chain based on a transient graph, and consequently admits no measures of maximal entropy.
    \item The system $(\G, \hat{f}_{\text{MME}})$ is encoded by a countable strongly positive recurrent topological Markov chain, forcing the existence of a unique measure of maximal entropy.
\end{enumerate}

Both constructions first retract the dynamics to a sufficiently deep finite tree and then insert controlled countable Markovian structures. In the non-MME case the associated graph is arranged to be transient. In the MME case the graph is further modified by a careful inclusion of a finite high-entropy component which yields strong positive recurrence. 
Additionally, both perturbations are obtained by changing the original map locally on a small scale.

The paper is organized as follows. Section 2 recalls the necessary notions on dendrites and topological dynamics,  and later treats on the Markov and Markovian maps. Section 3 establishes the entropy-preserving correspondence between one-sided countable topological Markov chains and the Markovian maps used later. Section 4 gives the explicit approximation construction and proves the main theorem. Section 5 discusses possible extensions beyond the Gehman dendrite and formulates related open questions.

\section{Preliminaries}
We denote by $\mathbb{N}=\{0,1,\ldots\}$ the set of non-negative integers. For $k\in \mathbb{N}$ we write
$\mathbb{N}_k=\mathbb{N}\cap [k,\infty)$.

\subsection{Continua, topological graphs and dendrites}\label{the_family_of_trees}
A \emph{continuum} is a compact, connected metric space. A continuum is non-degenerate if it contains at least two points.
A \emph{dendrite} is a non-degenerate locally connected continuum without a subset homeomorphic to a simple closed curve. An \emph{arc} in a continuum $X$ is a set $\mathcal{I}\subseteq X$ that is a homeomorphic copy of the interval $[0,1]$. In other words, $\mathcal{I}$ is an arc if there exists a continuous bijection $\varphi\colon [0,1]\to \mathcal{I}\subset X$. In this situation, we call points $\varphi(0)$ and $\varphi(1)$ the \emph{endpoints} of $\mathcal{I}$. We will denote $\mathcal{I}=[x, y]$, where $x=\varphi(0), y=\varphi(1)$.

A \emph{topological graph} is a continuum $G$ composed of a finite number of arcs, any two of which either have a common endpoint or are disjoint. We distinguish the family of \emph{topological trees}, which are topological graphs with no simple closed curve. Clearly, topological trees lie in the intersection of topological graphs and dendrites.

Let $X$ be a dendrite. We write $\End(X)$ for the set of its endpoints (points $x\in X$ such that $X\setminus\{x\}$ is connected), $\Reg(X)=X\setminus \End(X)$ and $B(X)$ for its \emph{branch points}, meaning points $b\in X$ such that $X\setminus \{b\}$ has at least three connected components. 

In any dendrite $X$, $B(X)$ is always at most countable and is empty if and only if $X$ is an arc, while $\End(X)$ is always nonempty. An arc $A\subseteq X$ is a \emph{free arc} if all points in $A$ except possibly the endpoints of $A$ are not branch points in $X$. Let $b_1, b_2$ be two adjacent branch points in $X$ (where by \emph{adjacency} of two points in $X$ we mean nearest proximity with respect to the metric). We will say that an arc connecting $b_1, b_2$ is an \emph{edge} in $X$ and denote this arc by $[b_1, b_2]$. Since dendrites are uniquely arcwise connected, we can extend this notation to all points $x, y\in X$ and denote by $[x, y]$ a unique arc connecting $x, y$. To simplify the notation, we will denote $\Int([x, y])=(x, y)$. At the same time, we will use the standard notation $\overline{A}$ for the closure of any set $A$.

In this paper, we will examine maps on the Gehman dendrite $\G$, a continuum such that its endpoint set $\End(\G)$ is homeomorphic to the standard Cantor ternary set and each branch point $b\in B(\G)$ is of order 3.

\begin{theorem}[{\cite[Theorem 4.1]{Charatonik}}]\label{Gehman_unique}
    The Gehman dendrite is a unique (up to homeomorphism) continuum such that the endpoint set $\End(X)$ is homeomorphic to the standard Cantor ternary set and each branch point $b\in B(X)$ is of order 3.
\end{theorem}

The \textit{root} of $\mathcal{G}$ is a distinguished point of order 2. The Gehman dendrite can be pictured as an infinite binary tree where each branch point, as well as the root, has exactly one parent and each branch point has exactly two children. We use finite binary words to label the root and branch points in $\G$. First, we label the root with the empty word $\lambda$, that is, we  denote by $b_\lambda$ the root of $\G$. For any branch point (or the root itself) $b_\omega$ of $\G$ labeled with binary word $\omega$, we label its ``left" child as $\omega0$ (append $0$ to the word $\omega$) and we label its  ``right" child as $\omega1$ (append $1$ to the word $\omega$). We define the \emph{$n$-th level} in $\G$ to be the union of all edges $[b_\omega, b_{\omega'}]$, where $\omega'\in\{0, 1\}^n$ and denote it by $\mathcal{L}_n$. In particular, the edges $[b_\lambda, b_0], [b_\lambda, b_1]$ form the first level $\mathcal{L}_1$, the four edges $[b_0, b_{00}], [b_0, b_{01}], [b_1, b_{10}], [b_1, b_{11}]$ form $\mathcal{L}_2$ and so on. Observe that each level $\mathcal{L}_n$ is composed of $2^n$ edges in $\G$.

\begin{figure}[h]
	\centering
	\begin{tikzpicture}[scale=0.95]
		\begin{scope}[on background layer]
			\fill[blue!7,rounded corners=3pt]
			(0,0.35) -- (-6.5,-3.55) -- (6.5,-3.55) -- cycle;
		\end{scope}
		\node[lab,blue!55!black] at (-0,-2.2) {$T_2=\Lev_1\cup\Lev_2$};
		
		\node[rt,label={[lab]above:{$b_\lambda$ \scriptsize(order $2$)}}] (r) at (0,0) {};
		\node[bp,label={[lab]left:$b_0$}]  (b0) at (-4.2,-1.8) {};
		\node[bp,label={[lab]right:{$b_1$ \scriptsize(order $3$)}}] (b1) at ( 4.2,-1.8) {};
		\node[bp,label={[lab]left:$b_{00}$}]  (b00) at (-5.9,-3.2) {};
		\node[bp,label={[lab]right:$b_{01}$}] (b01) at (-2.5,-3.2) {};
		\node[bp,label={[lab]left:$b_{10}$}]  (b10) at ( 2.5,-3.2) {};
		\node[bp,label={[lab]right:$b_{11}$}] (b11) at ( 5.9,-3.2) {};
		\foreach \x/\n in {-6.5/c0,-5.3/c1,-3.1/c2,-1.9/c3,1.9/c4,3.1/c5,5.3/c6,6.5/c7}
		\node[bp] (\n) at (\x,-4.1) {};
		
		\draw[edge1] (r)--(b0); \draw[edge1] (r)--(b1);
		\draw[edge2] (b0)--(b00); \draw[edge2] (b0)--(b01);
		\draw[edge2] (b1)--(b10); \draw[edge2] (b1)--(b11);
		\draw[edge3] (b00)--(c0); \draw[edge3] (b00)--(c1);
		\draw[edge3] (b01)--(c2); \draw[edge3] (b01)--(c3);
		\draw[edge3] (b10)--(c4); \draw[edge3] (b10)--(c5);
		\draw[edge3] (b11)--(c6); \draw[edge3] (b11)--(c7);
		\foreach \n in {c0,c1,c2,c3,c4,c5,c6,c7}{
			\draw[edged] (\n)-- ++(-0.35,-0.7);
			\draw[edged] (\n)-- ++( 0.35,-0.7);
		}
		
		\draw[brc] (7.15,0.0) -- (7.15,-1.8) node[midway,right=4pt,lab]{$\Lev_1$ (length $\sim 4^{-1}$)};
		\draw[brc] (7.15,-1.8) -- (7.15,-3.2) node[midway,right=4pt,lab]{$\Lev_2$ ($\sim 4^{-2}$)};
		\draw[brc] (7.15,-3.2) -- (7.15,-4.1) node[midway,right=4pt,lab]{$\Lev_3$ ($\sim 4^{-3}$)};
		
		\draw[brc,decoration={mirror}] (-6.9,-4.95) -- (6.9,-4.95)
		node[midway,below=5pt,lab]{$\End(\G)\;\cong\;$ Cantor set};
	\end{tikzpicture}
	\caption{ Schematic view of: the Gehman dendrite $\G$, levels $\Lev_n$, trees $T_n$, labeling $b_\omega$ and scaling distance on $\Lev_n$.}
	\label{fig:gehman}
\end{figure}

\begin{remark}\label{retraction_map}
	Using this notation, we will also specify a special family $\mathcal{T}$ of subtrees of the dendrite $\G$, where each tree $T_n=\bigcup_{k=1}^n\mathcal{L}_k$. In other words, each $T_n\in\mathcal{T}$ is a full binary tree of height $n\in\mathbb{N}_1$ and $\G, T_1, T_2,\dots$ shares the same root. Moreover, for each $T_n\in\mathcal{T}$ there exists a standard retraction map $R_n\colon \G\to T_n$ (see \cite[Chapter X.3]{Nadler}) and for any $n\in\mathbb{N}_1$ the space $\overline{\G\setminus T_n}$ is a collection of $2^n$ homeomorphic copies of the Gehman dendrite.
	
	In the paper, we will repeatedly look for subsets of $\G$ in the form of topological trees. For simplicity and clarity of reasoning, in each case we will assume that the root of every subtree called out in $\G$ is the point that is closest to the root $b_\lambda$ of $\G$.
\end{remark}

Following the labeling scheme, we associate an arc length metric $d_{\G}\colon \G\times \G \to [0,\infty)$. Firstly, we set distances on each level in $\G$. We endow edges from level $n$ with the standard arc length metric inherited from the unit interval $[0,1]$, rescaled by a factor of $\frac{1}{4^n}$. Then we compute the distance $d(x, y)$ as the length of the unique arc $[x, y]$, with respect to the distances set on each level.

\begin{remark}
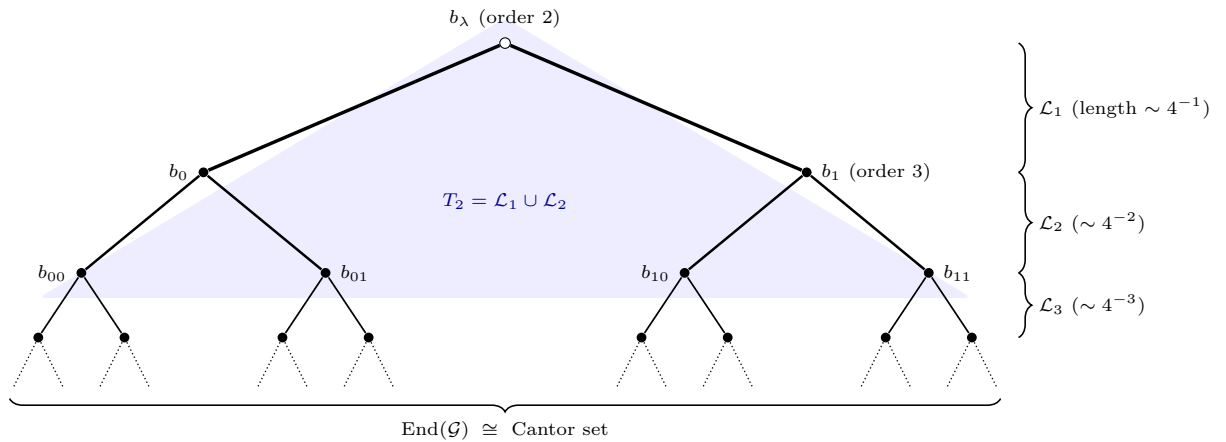
    In the later stages of the paper, we will present our reasoning with respect to arbitrary edges $[b, b']$. Using the notation presented above, we will assume that $d_{\G}(b_\lambda, b)<d_{\G}(b_\lambda, b')$. In other words, the edge $[b, b']$ is oriented with respect to the root of $\G$.
\end{remark}

\subsection{Notions from Topological Dynamics}

Let $(X,f)$ be a \emph{topological dynamical system} (a TDS for short). The space $X$ is compact, metric, and $f\colon X\to X$ is a continuous surjection. Any subset $A\subseteq X$ such that $f^{-1}(A)= A$ is called an \textit{invariant subset}. Moreover, if $A$ is nonempty and closed, then we call it a \emph{subsystem} of $(X,f)$. If $A$ is a subsystem, then $(A,f|_A)$ is also a TDS. Moreover, we denote $\mathcal{O}^+(x)=\{f^{n}(x)\colon n\in\mathbb{N}\}$.

\begin{definition}\label{periodic_recurrent_def}
    Let $(X, f)$ be a TDS. We call $x\in X$ a \textit{periodic point} for $f$ if there exists $n\in\mathbb{N}_1$ such that $f^{n}(x)=x$. We call $y\in X$ a \textit{recurrent point} if for every neighborhood $U\ni y$ there exists $n_{U}\in\mathbb{N}_1$ such that $f^{n_{U}}(y)\in U$.
\end{definition}

Let $(X,f)$ and $(Y,g)$ be two TDS. We say $(Y,g)$ is a \emph{factor} of $(X,f)$ (and we call $(X,f)$ an \emph{extension} of $(Y,g)$) if there exists a \emph{factor map}, that is, a continuous surjection $\varphi\colon X\to Y$ such that $\varphi\circ f=g\circ\varphi$. If $\varphi$ is a homeomorphism, then we say that $(X,f)$ and $(Y,g)$ are \emph{conjugate}. 

\begin{definition}
A TDS $(X,f)$ is: 

\begin{itemize}
\item \emph{transitive} if for every pair of nonempty open subsets $U,V\subseteq X$ there exists $n\in\mathbb{N}$ such that $f^{n}(U)\cap V\neq\emptyset$;
\item \emph{totally transitive} if for each $n\in\mathbb{N}$ the map $f^n$ is transitive;
    \item 
\emph{(topologically) mixing} if for every pair of nonempty open subsets $U,V\subseteq X$ there exists $n_0\in\mathbb{N}$ such that for all $n\geq n_0$ we have $f^{n}(U)\cap V\neq\emptyset$;
\item \emph{exact} if for each nonempty and open subset $U\subseteq X$ there exists $n\in\mathbb{N}$ such that $f^{n}(U)=X$;
\end{itemize}
\end{definition}

We denote by $\htop(X, f)$ the \emph{topological entropy} of TDS $(X, f)$. For definition and further details, see \cite[Chapter 7, 8]{Walters}. Topological entropy $\htop(X, f)\in[0,\infty]$ is a numerical invariant  of conjugacy of TDS.

We state here only these properties of entropy that we need to use in our presentation. We will use these facts without specifically calling them in our reasoning.
\begin{theorem}\label{entropy_properties}
   The topological entropy $\htop(X, f)$ has the following properties:
   \begin{enumerate}
       \item If $(Y, g)$ is a factor of $(X, f)$, then $\htop(Y, g)\leq \htop(X, f)$. In particular, $\htop(X, f)=\htop(Y, g)$ if $f$ and $g$ are conjugate.
       \item If $A\subseteq X$ is a subsystem of $(X, f)$, then $\htop(A, f|_A)\leq \htop(X, f)$.
       \item If $n\ge 1$, then $\htop(X, f^n)=n\htop(X, f)$. In particular, if $f^n=\mathrm{Id}$, then $\htop(X, f)=0$.
              \item Let $I$ be a nonempty set of indices. If for every $i\in I$ the set $X_i\subseteq X$ is a subsystem of $(X,f)$ and $X=\bigcup_{i\in I} X_i$, then (by the variational principle)
       \begin{equation*}
           \htop(f)=\sup_{i\in I}\htop(f|_{X_i}).
       \end{equation*}
   \end{enumerate}
\end{theorem}

\begin{theorem}[{\cite[Corollary 2.2]{Misiurewicz_hausdorff}}] \label{lipschitz_entropy_graf}
Let $(X,f)$ be a TDS on a metric space $(X,d)$. If $f\colon X\to X$ is $L$-Lipschitz for some $L>1$, that is, if $d(f(x),f(y))\le Ld(x,y)$ for every $x,y\in X$,
then $\dim_H(X)\cdot \log (L)\geq \htop(X, f)$, where $\dim_H(X)$ is the Hausdorff dimension of $X$.
\end{theorem}

It is well-known that for the pair $(\G, d_{\G})$ we have $\dim_H(\G)=1$ (see e.g. \cite[Claim 3]{Entropy_paradox_dendryty}).

\begin{corollary}
    Let $(\G, f)$ be a TDS on the metric space $(\G, d_{\G})$. If $f\colon \G \to\G$ is $L$-Lipschitz, then $\htop(\G, f)\leq\log L$.
\end{corollary}

The next theorem is a special case of {\cite[Theorem C]{Positive_entropy_Gehman}}. We present it in a form that suits best our considerations.

\begin{theorem}\label{gehman_transitive_positive_entropy}
    A transitive map on the Gehman dendrite has positive topological entropy and dense set of periodic points.
\end{theorem}

Another way of considering the notion of entropy is with respect to a measure. We will briefly present the necessary notions needed for our presentation.
\begin{definition}
    Let $(X, f)$ be a TDS and $\mu$ be a probability measure on $X$. We say that $\mu$ is \emph{$f$-invariant} (or generally invariant) if $\mu(f^{-1}(U))=\mu(U)$ for all measurable sets $U\subset X$. In this setting, we also say that $f$ is \emph{$\mu$-preserving} (or measure preserving).
\end{definition}

\begin{remark}
    If $(X, f)$ is a TDS, then we assume that $\mu$ is a measure defined on the Borel sigma algebra $\mathcal{B}(X)$. Sometimes the system $(X, f)$ may not have a topological structure. In that case we call $(X, B, f, \mu)$ a \emph{measurable dynamical system}, where $B$ is a sigma-algebra of subsets of $X$, $f\colon X\to X$ is a measurable map and $\mu$ is a finite $f$-invariant measure.
\end{remark}

\begin{definition}
    Let $(X, f)$ be a TDS. We say that an $f$-invariant measure $\mu$ is \emph{ergodic} if for each Borel set $A$ such that $f^{-1}(A)=A$ we have $\mu(A)\in\{0,1\}$.
\end{definition}

\begin{definition}\label{generic_assumptions}
    Let $(X, f)$ be a TDS and $\mu$ an ergodic measure for $f$. We say that a point $x\in X$ is \textit{generic} if we have (in the  weak-* topology:)
    \begin{equation*}
        \lim_{n\to\infty} \frac{1}{n}\sum_{k=0}^{n-1}\delta_{f^k(x)}=\mu.
    \end{equation*}
\end{definition}
\begin{remark}
    Generic points are recurrent points and have dense orbits in measure's support.
\end{remark}

We denote by $h_\mu(X, f)$ the measure-theoretic entropy of $(X,f)$ (or in general for a measurable system $(X, B, f, \mu))$ with respect to an invariant measure $\mu$. For definition and further details, see \cite[Chapter 4 \& 8]{Walters}.

\begin{remark}
    It is common to say that the measure $\mu$ has some (finite/infinite or zero/positive) entropy whenever it is clear what system we are referring to.
\end{remark}
\begin{remark}
    We denote by $\mathcal{M}_f(X)$ the set of $f$-invariant measures on $X$ and by $\mathcal{M}_f^+(X)$ the set of ergodic $f$-invariant measures on $X$ with positive entropy.
\end{remark}
\begin{definition}
    Let $(X, f), (Y, g)$ be two TDS and $\mu_X, \mu_Y$ be two ergodic measures on $X, Y$ respectively. We say that $f, g$ are \textit{measurably isomorphic} if there exist invariant subsets $X'\subset X, Y'\subset Y$ with $\mu_X(X')=1=\mu_Y(Y')$ and a measure preserving bijection $\varphi\colon X'\to Y'$ such that $\varphi\circ f=g\circ \varphi$ on $X'$.
\end{definition}

\begin{theorem}[{\cite[Theorem 4.11]{Walters}}]\label{entropy_iso}
    If $(X, f), (Y, g)$ are measurably isomorphic, then $h_{\mu_X}(X, f)=h_{\mu_Y}(Y, g)$.
\end{theorem}

\begin{theorem}[Variational Principle]\label{VP}
    Let $(X, f)$ be a TDS and $\mathcal{M}(X)$ be set of all invariant measures on $X$. Then
    \begin{equation*}
        \htop(X, f)=\sup\{h_\mu(X, f)\colon \mu\in \mathcal{M}_f(X)\}.
    \end{equation*}
    Moreover, the same assertion holds true if the supremum is taken over the set of ergodic measures on $(X, f)$.
\end{theorem}

\begin{definition}
    Let $(X, f)$ be a TDS. We say that an ergodic measure $\mu$ is a \emph{measure of maximal entropy} (MME) if $\htop(X, f)=h_\mu(X, f)$.
\end{definition}

\subsection{Markov and Markovian maps}
We say that a continuous map $f\colon X\to Y$ between topological spaces is \emph{monotone} if for every $y\in Y$ the preimage $f^{-1}(y)$ is a connected subset of $X$.

A map $f\colon T\to T$ on a topological tree $T$ is \textit{$P$-monotone} if $P\subseteq T$ is a finite set  containing all branching points of $T$ such that for each connected component $C$ of $T\setminus P$ the map $f\colon \overline{C}\to T$ is monotone. We call $\overline{C}$ a \emph{$P$-basic interval} of $f$. Observe that each connected component $C$ of $T\setminus P$ must be an open subset of $T$, as every branching point of $T$ belongs to $P$. Consequently, every $P$-basic interval is a free arc. 

\begin{definition}\label{piecewise_linearity_def}
    Let $X, Y$ be dendrites and $I\subset X$ be a free arc. We say that a map $f\colon X\to Y$ is \emph{linear on $I$} if it is monotone and there exists a constant $s\geq 0$ such that $\diam(f(J))=s\cdot \diam(J)$ for every arc $J\subset I$. We call $I$ a \emph{linearity interval} for $f$ and $s_{f}(I)=s$ the \emph{slope} of $f$ on $I$. We say that $f$ is \emph{piecewise linear} (\emph{$\sigma$-linear}) if $X$ can be written as a finite union (contains a countable dense union) of free arcs such that $f$ is linear on each of those arcs. Moreover, we say that a $\sigma$-linear map $f$ has \emph{constant slope} if for each of those arcs the slope is the same.
\end{definition}

 \begin{remark}\label{rem:linearity}
    For simplicity in our considerations, we will assume that the linearity intervals are by default maximal in the sense of inclusion. In other words, we consider an interval $I$ to be a linearity interval for $f$ if it is a linearity interval and is not properly contained in any other linearity interval.
\end{remark}

\begin{remark}
    Let $f\colon X\to Y$ be a piecewise linear ($\sigma$-linear) map. We say that the slope $s_f$ of $f$ is bounded from below by some constant $c\in\mathbb{R}$, which we denote $s_f\geq c$, whenever $$\inf\{s_f(I)\colon\: I\textnormal{ - linearity interval for }f\}\geq c.$$
    Similarly, we say that the slope $s_f$ of $f$ is bounded from above by some constant $C\in\mathbb{R}$ if $\sup\{s_f(I)\colon\: I\textnormal{ - linearity interval for }f\}\leq C$ and we denote it $s_f\leq C$ when this occurs.
\end{remark}

\begin{definition}
    Let $[x, y], [x', y']\subset X$ be arcs in a dendrite $X$ and $M\in\mathbb{N}_1$.
    We say that a piecewise linear map $f\colon [x, y]\to [x', y']$ is an \emph{$M$-fold} if $[x, y]=[j_1, j_2]\cup\dots \cup[j_M, j_{M+1}]$ with $\diam([j_k, j_{k+1}])=\frac{\diam([x, y])}{M}$, $f$ is linear on $[j_k, j_{k+1}]$    with alternating endpoint image and the image $f(x)=x'$ and $f([j_k, j_{k+1}])=[x', y']$ for $k=1, \dots, M$.

    We will say that we set an $M$-fold between two arcs $[x, y], [x', y']$ by defining an $M$-fold map $f\colon[x, y]\to [x', y']$. Note that a $1$-fold is a homeomorphism induced by identity.
\end{definition}

\begin{definition}
We say that a $P$-monotone map $f\colon T\to T$ on a topological tree $T$ is \emph{$P$-linear} if it is linear on each $P$-basic interval. Moreover, if $f(P)\subset P$, then we say that $f$ is a \emph{$P$-Markov} map. We say that $(T, f)$ is $P$-Markov TDS if $f$ is a $P$-Markov map.
\end{definition}

We extend here the notion of $P$-Markov maps to the Gehman dendrite $\G$.

\begin{definition}
    Let $P\subset \G\setminus\End(\G)$ be a countable set.
    We say that a continuous map $f\colon \G\to \G$ is \emph{$P$-Markovian} if $B(\G)\subset P$ and
    \begin{enumerate}
        \item $f(P)\subset P$,
        \item $f$ is $\sigma$-linear with respect to $P$, i.e. the family $$\overline{\mathcal{C}}=\{\overline{C}\colon\; C\subset\G\setminus P,\; C\textnormal{ - connected component}\}$$ is a countable set such that $f$ is linear on each one of them and the union $$\bigcup_{C\in\:\overline{\mathcal{C}}} C=\G\setminus \End(\G),$$
        \item for each $b_1, b_2\in B(\G)$ the set $[b_1, b_2]\cap P$ is finite.
    \end{enumerate}
\end{definition}

\begin{remark}\label{Markov_notation_extended}
    Note that there might be multiple finite (countable) sets $P$ such that $f$ is $P$-Markov ($P$-Markovian). We say that $f$ is \emph{Markov} (\emph{Markovian}) if it is $P$-Markov ($P$-Markovian) for some finite (countable), yet not explicitly provided
 set $P$. In this case, we will refer to $P$-basic intervals by calling them basic intervals.

\end{remark}

We say that a topological dynamical system $(\G, f)$ is Markovian if the map $f\colon \G\to \G$ is Markovian. Let $(\G, f)$ be a Markov (Markovian) TDS. We will define $G_f=(V(G_f), E(G_f))$ to be a graph associated to $(\G, f)$, where each vertex $v_{[p, q]}\in V(G_f)$ is identified with a basic interval $[p, q]$ and the set $E(G_f)$ of directed edges is given by a \textit{covering relation} $$v_{[p_1, p_2]}\rightarrow v_{[q_1, q_2]} \iff [q_1, q_2]\subset f([p_1, p_2]).$$ We will call $G_f$ the \textit{Markov graph} (\textit{Markovian graph}) for $f$. We say that a directed graph $G_f$ is \emph{strongly connected} if there exists a path in $G_f$ between each pair of vertices from $V(G_f)$. It is well-known that $G_f$ is strongly connected if and only if $f$ is transitive.

\begin{lemma}\label{finite_in_out}
    Let $(\G, f)$ be a Markovian TDS and $G_f$ its Markovian graph. Then for each vertex $v\in V(G_f)$ we have that there are finitely many $w\in V(G_f)$ such that either $(v, w)\in E(G_f)$ or $(w, v)\in E(G_f)$.
\end{lemma}
\begin{proof}
    Let $[p, q]$ be a basic interval for $f$ and $v$ be the vertex in $G_f$ corresponding to $[p, q]$.  Observe that $V(G_f)$ is countable. Indeed, the dendrite $\G$ is composed of countably many edges $[b_\omega, b_\omega']$ and each such edge contains a finite number of basic intervals. Assume first that there exist countably infinitely many $w_1, w_2,\dots\in V(G_f)$ such that $(v, w_n)\in E(G_f)$. Then the image $f([p, q])=\bigcup_{n=1}^{\infty}[p_n, q_n]$, where $[p_i, q_i]\neq[p_j, q_j]$, each $[p_n, q_n]$ is a basic interval and $\bigcup_{n=1}^{\infty}[p_n, q_n]$ is necessarily an arc, since $f$ is linear on $[p, q]$. It follows that $f(p)\in \End(\G)$ or $f(q)\in \End(\G)$.  Indeed, $f([p, q])$ is a linear stretch over $\G$ and contains infinitely many basic intervals, while each level in $\G$ contains only a finite number of basic intervals.
    Since for any set $P$ we have that $P\subset \G\setminus \End(\G)$ and $f(P)\subset P$, this is a contradiction.

    Assume that there are countably infinitely many $w_1,w_2,\dots \in V(G_f)$ such that $(w_n, v)\in E(G_f)$. Observe that if $[p_n, q_n]$ are basic intervals corresponding to $w_n$, then $\diam([p_n, q_n])\to 0$, as we have that $B(\G)\subset P$ and the diameters of free arcs forming level $\mathcal{L}_n$ in $\G$ decrease to zero with $n\to\infty$ and each level contains only finitely many basic intervals. Thus, there are $x_n,y_n\in[p_n,q_n]$ with $\{f(x_n),f(y_n)\}=\{p,q\}$, which implies that $d(x_n,y_n)\to0$ while $d(f(x_n),f(y_n))=d(p,q)>0$.
\end{proof}

\begin{remark}
    While the notion of ``Markovian'' systems can be regarded as closely resembling the ``countable Markov'' systems found in e.g. \cite{Bobok2016, Bobok2019, Buzzi2006, Ruette_natural_extension}, there are crucial structural differences that warrant a distinct terminology. In our setting, the set $\End(\G)$ is strictly disjoint from the set $P$, and the latter one is locally finite - properties clearly violated e.g. by the construction found in \cite{Ruette_natural_extension}. A key consequence we have just shown is that the Markovian graph is then locally finite as well. Consequently, these maps behave much like finite Markov maps adapted to countable structures, such as the Gehman dendrite.
\end{remark}

\begin{corollary}
    Let $(\G, f)$ be a Markovian TDS. Then $f$ is not exact, and we have that $(\End(\G), f|_{\End(\G)})$ is a subsystem.
\end{corollary}

\begin{proof}
Let $J$ be an arc between any two regular points. Then $f(J)$ is not degenerate, and thus cannot contain an endpoint by Lemma~\ref{finite_in_out}.
Similarly, by Lemma~\ref{finite_in_out}, for every point $z\in \End(\G)$ and every $n$ there is a neighborhood $U$ of $z$ such that $f(U)\cap \cup_{k\leq n}\mathcal{L}_k=\emptyset$,
so $f(\End(\G))\subset \End(\G)$.
\end{proof}

The following assertions will be crucial in our reasoning. We present them in a form that best suits our considerations.

\begin{theorem}[{\cite[Corollary 1.11]{Baldwin}, \cite[Lemma 8.1]{Entropy_paradox_drzewa}}]\label{connected_transitive}
    Let $f\colon T\to T$ be a Markov map on a topological tree $T$. Then $f$ is transitive if and only if the  Markov graph $G_f$ is strongly connected and is not a cycle.
\end{theorem}

\begin{theorem}[{\cite[Theorem 2.2]{Alseda}, \cite[Theorem 2.1]{Entropy_paradox_drzewa}}]\label{totally_transitive_subtrees_Markov}
    Let $f\colon T\to T$ be a transitive Markov map on a topological tree $T$. Then either $f$ is totally transitive or there exists $(T_k)_{k=1}^n$ - a finite family of nondegenerate subtrees such that
    \begin{enumerate}
        \item $T=\bigcup_{k=1}^n T_k$,
        \item $T_{k_1} \cap T_{k_2}=\End(T_{k_1})\cap \End(T_{k_2})$ for $k_1\neq k_2$,
        \item $f(T_k)=T_{k+1}$ with $f^{n}|_{T_k}$ - transitive.
    \end{enumerate}
\end{theorem}

\begin{theorem}[{\cite[Corollary 7.6]{Entropy_paradox_drzewa}}]\label{totally_transitive_exact}
    Let $f\colon T\to T$ be a totally transitive Markov map on a topological tree $T$. Then $f$ is exact.
\end{theorem}

\section{Markovian dynamics and measures of maximal entropy}

In the following section we aim to establish relations between a mixing Markovian dynamical system $(\G, f)$, its Markovian graph $G_f$ and a one-sided topological Markov chain $(\Gamma^{+}_{G_f}, \sigma)$ defined with respect to $G_f$. Firstly, we will present a well-known classification framework for Markovian graphs as well as classical results on the existence and uniqueness of maximal measures for two-sided topological Markov chains defined on these graphs. Then we will proceed with establishing a similar framework for the case of one-sided topological Markov chains. Since mixing Markovian maps are non-invertible, the approach for two-sided topological Markov chains is not valid in our considerations. We will finish by providing results on the existence and uniqueness of measures of maximal entropy for mixing Markovian maps based on the properties of their Markovian graphs.

\subsection{Theory on graph representation of countable topological Markov chains}\label{Markov_graph_theory}
The following theory is due to 
Gurevich \cite{Gurevich_Zargaryan,Gurevich_infinite_graphs_no_MME, Gurevich_MME_iff},
Ruette \cite{Ruette}, 
Salama \cite{Salama_entropy_R},
Vere-Jones \cite{Vere} and others. Our presentation follows \cite{Ruette}. (Note that the original papers by Gurevich are in Russian. We state and cite them based on the contents of \cite{Ruette}.)

Let $G$ be an oriented graph with a countable set of vertices $V(G)$. Assume that for each pair of vertices $u, v\in V(G)$ there is at most one arrow $u\to v$. We will call a \emph{path} in $G$ a sequence of vertices $(u_0, \dots, u_n)$ such that $u_{k}\to u_{k+1}$ for $k=0,\dots n-1$. A path is called a \emph{loop} if $u_0=u_n$.

We will write $H\subset G$ for all subgraphs $H$ of $G$. If $W$ is a subset of $V(G)$, then the complement of $V(G)\setminus W$ is denoted by $W^c$. Note that in this setting $W, W^c$ determine subgraphs of $G$, where we consider a subset of $E(G)$ composed of all arrows between two vertices from either $W$ or $W^{c}$.

In \cite{Vere}, Vere-Jones presented a classification of strongly connected oriented graphs using the following quantities:
\begin{enumerate}
    \item $p_{uv}^G(n)$ - the number of paths $(u=u_0, u_1, \dots, u_{n-1}, u_n=v)$ of length $n$ between $u, v$. We set $R_{uv}(G)$ to be the radius of convergence of series $\sum p_{uv}^G(n)z^n$.
    \item $f_{uv}^G(n)$ - the number of paths $(u=u_0, u_1, \dots, u_{n-1}, u_n=v)$ between $u, v$ with $u_k\neq v$ for $k\neq n$. We set $L_{uv}(G)$ to be the radius of convergence of series $\sum f_{uv}^G(n)z^n$.
\end{enumerate}
For a strongly connected oriented graph $G$ the value $R_{uv}(G)$ does not depend on $u, v$ and thus can be denoted by $R(G)$ \cite{Vere}. Vere-Jones proposed the following three classes of graphs: transient, null recurrent and positive recurrent. (For further details on this classification, as well as detailed comparisons on behaviours of the above mentioned quantities, see e.g. \cite[Table 1]{Ruette}.)

Let $G$ be a strongly connected oriented graph. Define the set
\begin{equation*}
    \Gamma_G=\{(v_n)_{n\in \mathbb{Z}}\colon v_n\to v_{n+1}\textnormal{ in }G\}\subset (V(G))^{\mathbb{Z}}.
\end{equation*}
The set $\Gamma_G$ contains all two-sided infinite paths in the graph $G$. The topological Markov chain on $G$ is the system $(\Gamma_G, \sigma)$, where $\sigma$ is the standard shift map. We endow $V(G)$ with the discrete topology and induce topology on $\Gamma_G$ from $V(G)^\mathbb{Z}$. The space $\Gamma_G$ is compact only when $G$ is finite. However, we can consider the one point compactification $V(G)\cup\{\infty\}$, defining $\overline{\Gamma}_G$ as the closure of $\Gamma_G$ in $(V(G)\cup\{\infty\})^{\mathbb{Z}}$.

We define the graph entropy $h(G)$ based on the topological entropy of an associated topological Markov chain. Gurevich presents this definition of entropy in \cite{Gurevich_infinite_graphs_no_MME}:
\begin{equation*}
    h(G)=\sup\{\htop(\Gamma_H, \sigma)\colon H\subset G\textnormal{ is a finite subgraph}\}.
\end{equation*}

Note that if $G$ is a finite graph, then $h(G)=\htop(\Gamma_G, \sigma)$ since the latter is well defined. The graph entropy can be also computed combinatorially, as the exponential growth of the number of paths with fixed endpoints.

\begin{theorem}[{\cite[Proposition 1.2]{Ruette}}]\label{different_uv_same_limit_entropy}
    Let $G$ be a strongly connected oriented graph. Then for any vertices $u,v\in V(G)$ we have
    \begin{equation*}
        h(G)=\lim_{n\to\infty}\frac{1}{n}\log p_{uv}^G(n)=-\log R(G).
    \end{equation*}
\end{theorem}
Another viable approach for computing the entropy of $G$ is to consider the closure $\overline{\Gamma}_G$. Moreover, the Variational Principle with respect to $\Gamma_G$ yields the same value of entropy.
\begin{theorem}[{\cite[Theorem 1.3]{Ruette}}]\label{entropia_grafu_domkniecie_shiftu}
    Let $G$ be an oriented graph. Then
    \begin{equation*}
        h(G)=\htop(\overline{\Gamma}_G, \sigma)=\sup\{h_\mu(\Gamma_G, \sigma)\colon\mu\in\mathcal{M}_\sigma(\Gamma_G)\}.
    \end{equation*}
\end{theorem}

It remains the question when the supremum is attained, i.e. there exists a $\sigma$-invariant measure $\mu$ which attains the entropy $h(G)$ in the assertion above. The Vere-Jones classification allows us to give a definitive answer, relying purely on the characteristics of $G$.

\begin{theorem}[{\cite[Theorem 3.1]{Ruette}}]\label{MME_existence}
    Let $G$ be a strongly connected oriented graph of finite positive entropy. Then the topological Markov chain $(\Gamma_G, \sigma)$ admits an ergodic measure of maximal entropy if and only if the graph is positive recurrent. Moreover, this measure is unique.
\end{theorem}

We will now present a few results which help determining whether a graph $G$ is positive recurrent.

\begin{theorem}[{\cite[Proposition 2.2]{Ruette}}]\label{Salama_R<L}
    Let $G$ be a strongly connected oriented graph. If $G$ is transient or null recurrent, then $R(G)=L_{uu}$ for all vertices $u\in V(G)$. If there exists a vertex $u$ such that $R(G)<L_{uu}$, then $G$ is positive recurrent.
\end{theorem}
If $G$ has positive entropy, then the notion of positive recurrence is replaced by strong positive recurrence.
\begin{definition}
    We say that a strongly connected oriented graph $G$ is \emph{strongly positive recurrent} if $R(G)<L_{uu}$ for all vertices $u\in V(G)$.
\end{definition}
\begin{theorem}[{\cite[Theorem 2.7]{Ruette}}]\label{Ruette_subgraphs_SPR}
    Let $G$ be a strongly connected oriented graph of positive entropy. Then the following properties are equivalent:
    \begin{enumerate}
        \item the graph $G$ is strongly positive recurrent,
        \item there exists a vertex $u\in V(G)$ such that $R(G)<L_{uu}$,
        \item $G$ has no proper subgraph of equal entropy.
    \end{enumerate}
\end{theorem}

\begin{remark}
    Finite topological Markov chains always admit an MME (see also \cite{Parry}).
\end{remark}

Let $G$ be a strongly connected oriented graph with a  globally bounded number of incoming/outgoing edges for all vertices in $V(G)$ and let $R(G)$ be radius of convergence defined above.
Let $A$ be its adjacency matrix $A=(a_{vw})_{v,w\in V(G)}$ that is, $a_{vw}=1$ iff $v\to w$ and $a_{vw}=0$ otherwise. 
A vector $x=(x_v)_{v\in V(G)}$ with $x\ge 0$, $x\not\equiv 0$, is called
\emph{$R$-subinvariant} if $Ax\le R^{-1}x$ coordinatewise, that is
$\sum_{w}a_{vw}x_{w}\le R^{-1}x_{v}$ for all $v$, and \emph{$R$-invariant} if equality
holds for all $v$.

\begin{theorem}[{\cite[Theorem 6.2]{Seneta}}]\label{Seneta:Rec}
Let $G$ be a strongly connected oriented graph with a bounded number of outgoing edges and let
$A$ be the associated adjacency matrix. If $G$ is
positively recurrent, then it possesses an $R$-subinvariant vector that is unique up to a
positive scalar multiple, and this vector is $R$-invariant.
\end{theorem}

\begin{definition}\label{def:tuv}
    Let $G$ be a strongly connected oriented graph and $W\subset G$. Define $t_{uv}^W(n)$ as the number of paths $(u=u_0,\dots,u_n=v)$ such that $u_k\in V(W)$ for $k=1,\dots,n-1$.
\end{definition}
\begin{theorem}[{\cite[Theorem 3.2]{Ruette}}]\label{SPR_high_local_entropy}
    Let $G$ be a strongly connected oriented graph of finite positive entropy. Assume that there exists a finite strongly connected subgraph $W\subset G$ such that for all vertices $u, v\in V(W)$ we have that 
    \begin{equation*}
        \limsup_{n\to\infty}\frac{1}{n}\log t_{uv}^{W^c}(n)\leq h(W).
    \end{equation*}
    Then $G$ is positive recurrent.
\end{theorem}

\subsection{One-sided topological Markov chains and measures of maximal entropy}\label{one-sided}

The material in this subsection is standard, but the precise one-sided form needed later is not easy to quote directly from the literature. We therefore include a self-contained formulation. One may find related comments in e.g. \cite[Section 2.2]{Ruette_natural_extension}. This section consists of a single theorem and a corollary that follows directly from it. Since this framework is meant to be used as a black box, we present the necessary notions and invoke appropriate theorems only inside the proof.

\begin{theorem}\label{one-sided-two-sided}
    Let $(\Gamma_{G_f}^+, \sigma)$ be a one-sided topological Markov chain, i.e. 
    \begin{equation*}
    \Gamma^+_{G_f}=\{(v_n)_{n\in \mathbb{N}_1}\colon v_n\to v_{n+1}\textnormal{ in }G_f\}\subset (V(G_f))^{\mathbb{N}_1}
\end{equation*}
and $\sigma$ is the shift map. Assume that $(\overline{\Gamma}_{G_f},\sigma)$ has finite topological entropy.
Then there exists a bijection $\mu\mapsto\mu_+$ between invariant measures $\mu\in\mathcal{M}(\Gamma_{G_f})$ on the two-sided topological Markov chain $\Gamma_{G_f}$ and invariant measures $\mu_+\in\mathcal{M}(\Gamma^+_{G_f})$ on the one-sided topological Markov chain $\Gamma^+_{G_f}$. Moreover, this bijection preserves the measure-theoretic entropy.
\end{theorem}

\begin{proof}
The proof relies heavily on the notion of the natural extension for a one-sided topological Markov chain. An \emph{inverse limit} of a system $(\Gamma_{G_f}^+, \sigma)$ is the space
    \begin{equation*}
        \varprojlim\:(\Gamma_{G_f}^+, \sigma)_{n\in\mathbb{Z}}=\left\{
            \Tilde{x}\in(\Gamma^+_{G_f})^\mathbb{Z}\colon \sigma(\Tilde{x}_n)=\Tilde{x}_{n+1}
        \right\},
    \end{equation*}
which we endow with a product topology, i.e. the smallest topology such that all projections $\pi_{m}\colon \varprojlim\:(\Gamma_{G_f}^+, \sigma)_{n\in\mathbb{Z}}\to (\Gamma^+_{G_f}, \sigma)_m$ are continuous. Followingly, we define a \emph{natural extension} of the system $(\Gamma_{G_f}^+, \sigma)$ as the inverse limit $\varprojlim\:(\Gamma_{G_f}^+, \sigma)_{n\in\mathbb{Z}}$ together with the Borel sigma algebra generated by the induced topology and a standard shift map denoted by $\Tilde{\sigma}$.
    
\begin{remarknn}\label{commutative_shifts}
    Note that a natural extension is an extension if both the extended system and its natural extension are topological dynamical systems. The definition, as it has been presented, implies that each projection $\pi_m$ is a factor map, provided that both underlying spaces are compact.
\end{remarknn}

It is well-known that there exists a unique measure $\Tilde{\mu}$ on the natural extension $(\varprojlim\:(\Gamma_{G_f}^+, \sigma)_{n\in\mathbb{Z}}, \Tilde{\sigma})$ such that for all $m\in\mathbb{Z}$ we have $\Tilde{\mu}(\pi_m^{-1}(A))=\mu_+(A)$ for every measurable set $A\subset\Gamma^+_{G_f}$ (see e.g. {\cite[Chapter V.3, Theorem 3.2]{Parthasarathy}}). One can see that the measure $\Tilde{\mu}$ preserves invariance, i.e. $\mu_+$ is $\sigma$-invariant if and only if $\Tilde{\mu}$ is $\Tilde{\sigma}$-invariant. Indeed, this is rather a straightforward argument in view of Remark above, provided that the projections of measurable sets via maps $\pi_m$ are measurable. The latter assertion follows from the Arsenin-Kunungui Theorem (see e.g. \cite[Theorem 18.18]{Kechris}) in conjunction with Lemma \ref{finite_in_out}.
    
It is well-known that an association $\phi(\Tilde{x})=(\Tilde{x}_n(0))_{n\in\mathbb{Z}}$ defines a natural homeomorphism $\phi\colon\varprojlim\:(\Gamma_{G_f}^+, \sigma)_{n\in\mathbb{Z}}\to\Gamma_{G_f}$. Indeed, this follows from the fact that each coordinate in $\Tilde{x}$ is a valid path in one-sided topological Markov chain $\Gamma_{G_f}^+$. Therefore, there exists a natural bijection $\mu=\Tilde{\mu}(\phi^{-1})=\phi_*\Tilde{\mu}$ which preserves invariance, ergodicity and measure-theoretic entropy. Let us then examine the measure-theoretic entropies $h_{\mu_+}(\Gamma^+_{G_f}, \sigma)$ and $h_\mu(\Gamma_{G_f}, \sigma)$. We will do so with respect to the generating partitions (also called generators) for these dynamical systems.

For two partitions $\alpha, \beta$ for a system $(X, f)$ we denote $$\alpha\vee\beta=\{A\cap B\colon A\in\alpha, B\in\beta\}$$ and consequently for $n\in\mathbb{N}$ we denote $$\bigvee_{i=0}^n f^{-i}(\alpha)=\{A_0\cap f^{-1}(A_1)\cap\dots\cap f^{-n}(A_n)\colon A_0,\dots,A_n\in\alpha\}.$$

A countable measurable partition $\alpha$ is a \emph{generator (strong generator)} for an invertible (possibly noninvertible) measurable dynamical system $(X, B, f, \mu)$ if we have that $\bigvee_{n=k}^\infty f^{-n}(\alpha)=B$, where $\bigvee_{n=k}^\infty f^{-n}(\alpha)$ is the smallest sigma algebra containing $\bigcup_{n=k}^\infty f^{-n}(\alpha)$ with $k=-\infty$ ($k=0$). In such setting we define the entropy of a countable partition $\alpha$ as $H_\mu(\alpha)=\sum_{A\in\alpha}-\mu(A)\log\mu(A)$. Consequently, the measure-theoretic entropy of $(X, B, f, \mu)$ with respect to partition $\alpha$ is 
$$
h_\mu(\alpha, f)=\lim_{n\to\infty}\frac{1}{n}H_\mu(\bigvee_{i=0}^{n-1}f^{-i}(\alpha)).
$$

Observe that the topological dynamical system $(\overline{\Gamma}^+_{G_f}, \sigma)$, as a factor of $(\overline{\Gamma}_{G_f}, \sigma)$ whose entropy $\htop (\overline{\Gamma}_{G_f}, \sigma)<\infty$, necessarily has finite topological entropy. Hence, $\mu$ must have finite entropy. To add to that, $(\Gamma^+_{G_f}, \mathcal{B}(\Gamma^+_{G_f}), \sigma, \mu)$ is a standard probability space, since $\Gamma^+_{G_f}$ is a Polish space. Thus, there exists a (by default strong) finite entropy generator $\alpha$ by Krieger's Generator Theorem (see e.g. {\cite[Section 5.3, Theorem 3.2]{Petersen}}). Moreover, for any $m\in\mathbb{Z}$ the preimage $\phi\circ\pi_m^{-1}(\alpha)=\beta_m$ is a generator for the two-sided topological Markov chain $(\Gamma_{G_f}, \sigma)$, as $\Gamma_{G_f}$ is also Polish. One can see that the entropies of $\alpha, \beta_m$ match, as the measure $\tilde{\mu}$ acts as a ``measure preserving link" between $\mu_+$ and $\mu$. Therefore, $h_{\mu_+}(\Gamma^+_{G_f}, \sigma)=h_\mu(\Gamma_{G_f}, \sigma)$ by Sinai's Generator Theorem (see e.g. {\cite[Section 5.3, Theorem 3.1]{Petersen}}). \qedhere
\end{proof}

By Theorem~\ref{one-sided-two-sided}  there exists a bijective association $\mu_+\mapsto\mu$, in other words, each measure $\mu_+$ uniquely defines the measure $\mu$ and vice versa. This association preserves shift invariance and measure-theoretic entropy. Hence, these systems must have the same number of measures of maximal entropy.

\begin{corollary}\label{one_sided_shift_graph_VP}
    Let $G_f$ be a Markovian graph with finite positive entropy and $(\Gamma^+_{G_f}, \sigma)$ be a one-sided topological Markov chain defined by $G_f$. Then 
    \begin{equation*}
    h(G_f)=\htop(\overline{\Gamma}^+_{G_f}, \sigma)=\sup\{h_{\mu_+}(\Gamma^+_{G_f}, \sigma)\colon\mu_+\in\mathcal{M}_\sigma(\Gamma^+_{G_f})\}
    \end{equation*}
    and $(\Gamma^+_{G_f}, \sigma)$ admits a unique measure of maximal entropy if and only if the graph $G_f$ is positive recurrent.
\end{corollary}

\begin{proof}
    The assertion follows by combining the entropy preserving correspondence from Theorem \ref{one-sided-two-sided} with the countable graph criteria in Theorems \ref{entropia_grafu_domkniecie_shiftu} and \ref{MME_existence}.
\end{proof}

\begin{remark}
    Note that, following \cite{Ruette} we defined the topology on $\Gamma_G$ as induced by the discrete topology on $V(G)$. We silently extended this setting to $\Gamma_G^+$. However, one may always consider a compatible metric on $\Gamma_G^+$ e.g. $d(x, y)=2^{-\min\{i\in\mathbb{N}\colon x_i\neq y_i\}}$. For simplicity in future considerations, we will assume that the space $\Gamma_G^+$ comes with the metric $d$.
\end{remark}

\subsection{Dynamical properties of Markovian TDS}\label{MMEs_assertions}

Recall that in this section we assume that $(\G, f)$ is a mixing Markovian dynamical system and $G_f$ is the Markovian graph of $f$. Moreover, in $G_f$ each vertex $v_{[p, q]}$ is associated to a basic interval $[p, q]$ and each directed edge is based on the covering relation between basic intervals. Assume that the set $P_f\subset G$ is such that $f$ is $P_f$-Markovian.

\begin{lemma}\label{diameters_to_zero}
    For each $(v_{[p_n, q_n]})_{n=1}^\infty\subset \Gamma^+_{G_f}$ there exists a unique point $x\in \Reg(\G)$ such that $f^{n-1}(x)\in [p_n, q_n]$ for all $n\in\mathbb{N}_1$.
\end{lemma}

\begin{proof}
    Set $C_k=\bigcap_{n=1}^k f^{-n+1}([p_n, q_n])$. We claim that for every $k\in\mathbb{N}$ the set $C_k$ is connected. We will show it by induction. Since $C_1$ is clearly connected, assume that $C_k$ is connected and let us examine the set $C_{k+1}$. Observe that $I=f^{k-1}(C_k)\subset [p_{k}, q_k]$ and $I$ is connected. If $C_{k+1}$ is not connected, then it follows that $f(I)$ is not a connected subset of $[p_{k+1}, q_{k+1}]$ and hence $f$ is not linear over $[p_k, q_k]$, a contradiction. Thus, each $C_k$ is a closed subinterval of $[p_1, q_1]$.
    
    The sequence $(C_k)_{k=1}^\infty$ is nested. Consider the set $J=\bigcap_{k=1}^\infty C_k$ - a closed subset of $[p_1, q_1]$. If all $C_k$ are connected, then so is $J$. Assume that $\Int(J)\neq\emptyset$. Since $f$ is mixing, for any two basic intervals $[p, q], [r, w]$ there exists $k\in\mathbb{N}$ such that $f^k(J)\cap [p, q]\neq\emptyset$ and $f^{k}(J)\cap[r, w]\neq\emptyset$. On the other hand, by the definition of $C_k$ we have that $f^{n-1}(\Int(J))\subset [p_n, q_n]$ for $n\in\mathbb{N}$, which is a contradiction. Thus, $J=\{x\}$.
\end{proof}

\begin{definition}\label{phi_function}
    Let $\phi\colon (\Gamma^+_{G_f}, \sigma)\to (\G\setminus\End(\G), f|_{\G\setminus\End(\G)})$ be a function that associates a sequence $(v_{[p_n, q_n]})_{n\in\mathbb{N}}\subset \Gamma^+_{G_f}$ with the unique point $x\in \G\setminus\End(\G)$ such that $f^{n-1}(x)\in [p_n, q_n]$ for all $n\in\mathbb{N}_1$. We will call $(v_{[p_n, q_n]})_{n\in\mathbb{N}_1}$ the \emph{encoding} of $x$  and $\phi$ the \emph{decoding function}.
\end{definition}

\begin{lemma}\label{phi_cont_surj}
    The function $\phi$ is a continuous surjection.
\end{lemma}
\begin{proof}
    Since the map $f$ is Markovian, then for any point $x\in\G\setminus \End(\G)$ each of its images $f^{n-1}(x)\in[p_n, q_n]$ for $n\in\mathbb{N}_1$ and appropriate $[p_n, q_n]$. Thus, $\phi$ is surjective. Moreover, it must be continuous, as the critical observation is that for any sequence $(v_{[p_n,q_n]})_{n=1}^\infty$ we have $$\lim_{m\to\infty}\diam \bigcap_{n=1}^m f^{-n+1}([p_n, q_n])=0,$$
    which follows directly from Lemma \ref{diameters_to_zero}. Let $\epsilon>0$ and $n_\epsilon$ be such that the diameter $\diam\bigcap_{n=1}^{n_\epsilon} f^{-n+1}([p_n, q_n])<\epsilon$. Set $v\in\Gamma_{G_f}^+$ to be such that $v_n=v_{[p_n, q_n]}$ for $n=1,\dots,n_\epsilon$. There exists $\delta=2^{-n_\epsilon}$ such that if $d(v, w)<\delta$, then the initial positions $v_i=w_i$ for $n=1,\dots,n_\epsilon$. Hence, $\phi(v), \phi(w)\in \bigcap_{n=1}^{n_\epsilon} f^{-n+1}([p_n, q_n])$ and $d(\phi(v), \phi(w))<\epsilon$.
\end{proof}

\begin{definition}
    Let
    \begin{equation*}
        \mathrm{UEP}=\{x\in\Reg(\G)\colon \phi^{-1}(x)=\{(v_{[p_n, q_n]})_{n=1}^\infty \}\}
    \end{equation*}
    and $\mathrm{NEP}=\G\setminus(\End(\G)\cup \mathrm{UEP})$.
\end{definition}

\begin{lemma}
    Let $x\in\G\setminus\End(\G)$. Then there exist at most three encodings $u, w, v\in\Gamma^+_{G_f}$ such that $\phi(u)=\phi(v)=\phi(w)=x$.
\end{lemma}

\begin{proof}
    Observe that if $x\in\mathrm{UEP}$, then there exists only one $v\in\Gamma^+_{G_f}$ with $\phi(v)=x$, and we have that $f^{n-1}(x)\in(p_n, q_n)$, where $v_{[p_n, q_n]}=v_n$. In other words, the orbit $\mathcal{O}^+(x)\cap P_f=\emptyset$. Indeed, if this was not the case, then there would be $k\in\mathbb{N}_1$ such that $f^{k-1}(x)\in P_f$ would be an endpoint of two basic intervals (or three, if $f^{k-1}(x)\in B(\G)$). Let $[p, f^{k-1}(x)], [f^{k-1}(x), q]$ be those intervals. Then, the encodings $v, w$ given by formula
    \begin{equation*}
        \begin{split}
            &v_n=w_n=v_{[p_n, q_n]},\; f^{n-1}(x)\in(p_n, q_n),\;n=1,\dots,k-1\\
            &v_k=v_{[p, f^{k-1}(x)]},\;w_k=v_{[f^{k-1}(x), q]},\\
            &v_{n}=v_{[p_n, f^{n-1}(x)]}\;w_n=v_{[f^{n-1}(x), q_n]},\; n\geq k+1,
        \end{split}
    \end{equation*}
    satisfy $\phi(v)=\phi(w)=x$, where $p_n, q_n\in P_f$, $n\geq k+1$ are such that $v_n, w_n$ correspond to basic intervals in the images $f([p_{n-1}, f^{n-2}(x)])$ and $f([f^{n-2}(x), q_{n-1}])$, respectively. Since the map $f$ is linear on each of its basic intervals and $f^{k-1}(x)$ is an endpoint of a basic interval, there's always a unique $p_n$ such that $[p_n, f^{n-1}(x)]$ is a basic interval contained in $f([p_{n-1}, f^{n-2}(x)])$. The same holds true for the sequence $(q_n)_{n=k+1}^\infty$. This clearly contradicts the definition of $x\in\mathrm{UEP}$.

    Moreover, if $y\in\mathrm{NEP}$, then by a similar reasoning there exists either exactly two or exactly three encodings $u, v, w\in\Gamma^+_{G_f}$ such that $\phi(u)=\phi(v)=\phi(w)=y$.
\end{proof}

\begin{lemma}\label{DEP_countable}
    The set $\mathrm{NEP}$ is countable.
\end{lemma}
\begin{proof}
    Recall that $P_f$ is countable and invariant. Moreover, each basic interval covers a finite number of consecutive basic intervals. Thus, if $f([p, q])=\bigcup_{k=0}^{n} [p_k, p_{k+1}]$, then there are exactly $n$ points $x_1,\dots, x_n\subset (p, q)$ such that $f(x_k)\in P_f$. Let $\mathrm{P}_0=P_f$ and $\mathrm{P}_1=f^{-1}(P_f)\setminus P_f$ be the set of points that are mapped into $P_f$ at the first iteration of $f$, but do not belong to $P_f$. One can see that $\mathrm{P}_1$ is countable.

    Let $\mathrm{P}_n=f^{-n}(P_f)\setminus \bigcup_{k=0}^{n-1}\mathrm{P}_k$ be the set of points that are mapped into $P_f$ exactly at the $n$-th iteration of $f$. Since all of $\mathrm{P}_0,\dots, \mathrm{P}_{n-1}$ are countable, then $\mathrm{P}_n$ is countable. Thus $\mathrm{NEP}=\bigcup_{n=0}^\infty \mathrm{P}_n$ is countable.
\end{proof}

 \begin{lemma}\label{uep_homeo}
    The map $\phi$ induces a homeomorphism $\phi|_{\phi^{-1}(\mathrm{UEP})}$.
\end{lemma}
\begin{proof}
        One can see that $\phi|_{\phi^{-1}(\mathrm{UEP})}$ is a bijection, which follows straight from the definition of the set $\mathrm{UEP}$. Moreover, it is clearly continuous, as $\phi$ is continuous. Thus, it remains to show that its inverse $\phi^{-1}$ is also continuous. We once again rely on the fact that for any sequence $(v_{[p_n,q_n]})_{n=1}^\infty\in\Gamma^+_{G_f}$ the diameter $\lim_{m\to\infty}\diam \bigcap_{n=1}^m f^{-n+1}([p_n, q_n])=0$. Let $x\in\mathrm{UEP}$ and $v(x)\in \Gamma^+_{G_f}$ be its encoding. Let $\epsilon>0$ and $n_\epsilon$ be such that $2^{-n_\epsilon}<\epsilon$. Define $\delta>0$ to be such a constant that a ball $$B_{\mathrm{UEP}}(x, \delta)=B(x, \delta)\cap \mathrm{UEP}\subset \bigcap_{n=1}^{n_\epsilon} f^{-n+1}([p_n, q_n]).$$ Then, if $y\in B_\mathrm{UEP}(x, \delta)$, then $v(x), v(y)$ must agree on $n_\epsilon$ initial positions and hence $d(v(x), v(y))<\epsilon$.
\end{proof}

\begin{theorem}\label{measurable_isomorphism_uep_shift}
     Let $(\G, f)$ be a mixing Markovian dynamical system and $\phi$ be the decoding function. Then $\phi$ induces an isomorphism between $\mathcal{M}_f^+(\G\setminus \End(\G))$ and $\mathcal{M}_\sigma^+(\Gamma^+_{G_f})$. Moreover, this isomorphism preserves the measure-theoretic entropy.
\end{theorem}

\begin{proof}
	Fix $\mu \in \mathcal{M}_f^+(\G\setminus \End(\G))$.
    Observe that $\mu$ is supported either on $\mathrm{UEP}$ or $\mathrm{NEP}$, as these are disjoint invariant subsets covering $\G\setminus \End(\G)$. Since $\mathrm{NEP}$ is countable, it follows that entropy of any ergodic measure $\mu_{\mathrm{NEP}}$ supported on $\mathrm{NEP}$ is zero, because such a measure must be supported on a periodic orbit. Thus, the measure $\mu$ must be supported on the set $\mathrm{UEP}$.

Recall that the sets $\mathrm{UEP}, \phi^{-1}(\mathrm{UEP})$ are homeomorphic by Lemma \ref{uep_homeo}. Thus, we have that the systems $(\G, f, \mu)$ and $(\Gamma^+_{G_f}, \sigma, \nu)$, where $\nu=(\phi^{-1})_*\mu$ is supported on $\phi^{-1}(\mathrm{UEP})$ are measurably isomorphic. Likewise, if we fix any measure $\nu\in \mathcal{M}_\sigma^+(\Gamma^{+}_{G_f})$, it must be supported on $\phi^{-1}(\mathrm{UEP})$, as we have that $\phi^{-1}(\mathrm{NEP})$ is countable. Thus, $(\Gamma^+_{G_f}, \sigma, \nu)$ and $(\G, f, \mu)$, where $\mu=\phi_*\nu$, are measurably isomorphic. Therefore, Theorem \ref{entropy_iso} applies.
\end{proof}

\begin{corollary}\label{MME_FINAL_CONDITION}
    Let $(\G, f)$ be a mixing Markovian dynamical system with $$\htop(\End(\G), f|_{\End(\G)})<\htop(\G, f)<\infty$$ and $G_f$ be its Markovian graph. Then $h(G_f)=\htop(\G, f)$ and $(\G, f)$ admits a unique measure of maximal entropy if and only if $G_f$ is positive recurrent.
\end{corollary}
\begin{proof}
By Theorem \ref{measurable_isomorphism_uep_shift}, the map $\phi$ induces an isomorphism between the set $\mathcal{M}_f^+(\G\setminus \End(\G))$ and $\mathcal{M}_\sigma^+(\Gamma^+_{G_f})$.

Let $\mu$ be a measure on $\G$ such that $h_\mu(\G, f)>\htop(\End(\G), f|_{\End(\G)})$ provided by the variational principle. It follows that $h_\mu(\G, f)>0$ and thus, $\mu$ defines a measure $\nu$ on $\Gamma_{G_f}^+$ with equal entropy. Therefore, $\htop(\G, f)\leq \htop(\overline{\Gamma}^+_{G_f}, \sigma)$. On the other hand, every measure $\nu$ on $\Gamma_{G_f}^+$ with $h_{\nu}(\Gamma^{+}_{G_f},\sigma)>0$ defines a measure $\mu$ on $\G$ with equal entropy. Thus $\htop(\G, f)\geq \htop(\overline{\Gamma}^+_{G_f}, \sigma)$ and consequently $$\htop(\G, f)= \htop(\overline{\Gamma}^+_{G_f}, \sigma).$$
It follows that $h(G_f)=\htop(\G, f)$ by Corollary \ref{one_sided_shift_graph_VP}. Observe that if $(\G, f)$ admits a measure of maximal entropy, then this measure must be supported on $\G\setminus \mathrm{End}(\G)$. It follows from Corollary \ref{one_sided_shift_graph_VP} that $(\G, f)$ admits a unique MME if and only if $G_f$ is positive recurrent.
\end{proof}

\begin{corollary}\label{entropia_odwrocone_corollary}
    Let $(\G, f)$ be a mixing Markovian dynamical system and $G_f$ be its Markovian graph. If $\htop(\End(\G), f|_{\End(\G)})\leq h(G_f)$, then $\htop(\G, f)=h(G_f)$.
\end{corollary}

\begin{proof}
By Corollary \ref{one_sided_shift_graph_VP} we have $h(G_f)=\sup\{h_\mu(\Gamma_{G_f}^+, \sigma)\colon\; \mu\in\mathcal{M}_\sigma(\Gamma^+_{G_f})\}$. Moreover, by Theorem \ref{measurable_isomorphism_uep_shift}, there exists a bijection between ergodic measures on $\Gamma_{G_f}^+$ and $\G\setminus\End(\G)$ with positive entropy which preserves their entropy. We have that 
\begin{equation*}
    \begin{split}
        \htop(\G, f)=\sup\{&h_\mu(\G, f)\colon\; \mu\in\mathcal{M}_f(\G)\}=\\
        =\max\{\sup\{&h_\mu(f) : \mu\in \mathcal{M}_f(\G), \mu(\End(\G))=0\},\\
            &\htop(\End(\G), f|_{\End(\G)})\}.
    \end{split}    
\end{equation*}
Note that
    \begin{equation*}
        \begin{split}
        h(G_f)    &=\sup_{\mu\in\mathcal{M}_\sigma(\Gamma^+_{G_f})}\{h_\mu(\sigma)\}=\sup_{\mu\in\mathcal{M}_f(\G\setminus\End(\G))}\{h_\mu(f)\}=\\
        &=\sup\{h_\mu(f) : \mu\in \mathcal{M}_f(\G), \mu(\End(\G))=0\}.
        \end{split}
    \end{equation*}
    Thus, $\htop(\G, f)=h(G_f)$.
\end{proof}

\section{Main result}
We now turn to the approximation theorem. The purpose of the construction is twofold. First, it replaces the original map by a nearby Markovian map whose action is explicitly controlled on a countable family of free arcs. Second, it prescribes the recurrence type of the associated countable Markov graph, which in turn determines the existence of MMEs by the criteria recalled above. We prove the following.

\begin{theorem}\label{Markovian_approximation}
    Let $(\G, f)$ be a transitive TDS with $f(\End(\G))=\End(\G)$. Then for any $\gamma>0$ there exist two mixing Markovian maps $\hat{f}_\textnormal{NME}, \hat{f}_\textnormal{MME}$ on $\G$ such that
    \begin{enumerate}
        \item $d_{\sup}(f, \hat{f}_\textnormal{NME})<\gamma$ and  $d_{\sup}(f, \hat{f}_{\textnormal{MME}})<\gamma$,
        \item $(\G, \hat{f}_\textnormal{NME})$ admits no measures of maximal entropy,
        \item $(\G, \hat{f}_\textnormal{MME})$ admits a unique measure of maximal entropy.
    \end{enumerate}
\end{theorem}
\begin{proof}
We keep the proof constructive, since the same scheme will be used twice: first to build a transient graph, and then to modify it into a strongly positive recurrent graph. Let $\gamma>0$ be an arbitrary positive constant. We may assume that $\gamma <\diam(\G)$. Set $\epsilon=\frac{1}{22}\gamma$ and associate a positive $\delta<\epsilon$ such that $d_{\G}(x, y)<\delta\implies d_{\G}(f(x), f(y))<\epsilon$. Let  $\mathcal{U}$ be a finite open cover of $\G$ with $\mesh(\mathcal{U})\leq \delta/2$. By transitivity $V_\mathcal{U}:=\bigcap_{U\in\mathcal{U}}\bigcup_{n\ge0}f^{-n}(U)$ is a finite intersection of open dense sets, hence open dense itself (in particular nonempty).   By Theorem \ref{gehman_transitive_positive_entropy}, there exists a periodic point $o\in V_{\mathcal{U}}$ of some period $N$. Then by the definition, $o$ has the following property: for every $U\in\mathcal{U}$ there exists $0\leq n<N$ such that $f^{n}(o)\in U$. Moreover, there exists $M_0\in\mathbb{N}$ such that $T_{M_0}\in \mathcal{T}$ satisfies $\{f^{n}(o)\colon\;0\leq n< N\}\subset T_{M_0}$ (recall that the family $\mathcal{T}$ was defined in Section \ref{the_family_of_trees}).

    Let $\epsilon'=\dist(\End(T_{M_0}), \End(\G))$. Again, associate a positive $\delta'<\epsilon'$ such that if $d_{\G}(x, y)<\delta'$ then $d_{\G}(f(x), f(y))<\epsilon'$. There exists $M\in\mathbb{N}$ such that $M-M_0=d>1$ and the Hausdorff distance $$\dist_H(\End(T_M), \End(\G))<\delta',$$ where $T_M$ is again an element of $\mathcal{T}$. Set $R_M\colon\G\to T_{M}$ to be a standard retraction map (see Remark \ref{retraction_map}).
    \begin{step}
        Construction of $g\colon T_M\to T_M$ approximating $R_M\circ f$.
    \end{step}
    Let $h=R_M\circ f|_{T_M}$ and $P\subset T_M$ be a finite and $\delta/2$-dense set of points satisfying $$\{f^{n}(o)\colon 0\leq n < N\}\cup (B(\G)\cap T_M)\subset P.$$  
    We will define a map $g\colon T_M\to T_M$ with respect to $P$ and $h$. 
    
    Fix $p\in P\setminus \End(T_M)$. Set $g(p)=q$, where $q\in P$ is such that $d(h(p), q)$ is minimal. Then, for $e\in \End(T_M)$ set $g(e)=e'\in \End(T_M)$ such that $d(h(e), e')$ is minimal. Fix an arc $[p, q]$, where $p, q\in P$ are adjacent points and let $X_{[p, q]}\subset T_M$ be a minimal tree such that $\End(X_{[p, q]})\subset P$ and the union of balls in $T_M$ 
    \begin{equation*}
        \bigcup_{x\in[p, q]}B(h(x), \delta)\subset X_{[p, q]}\subset \bigcup_{x\in[p, q]}B(h(x), 2\delta).
    \end{equation*}
    Let $r\in P$ be such that $3\delta<\dist(r, [g(p), g(q)])<4\delta$. Note that such $r$ can always be found.  Indeed, the set $P$ is $\delta/2$-dense in $T_M$ with $\delta<\diam(\G)/22$, while, since the parameter $d>1$, we have that $\diam(T_M)>\frac{15}{16}\diam(\G)$ by the definition of the metric $d_{\G}$. Denote $\End(X_{[p, q]})=\{e_1,\dots, e_k\}$. We extend the map $g$ on $[p, q]$ in the following way: divide $[p, q]$ into $2k$ subintervals with pairwise disjoint interiors, i.e.
    $$[p, q]=[p, j_1]\cup[j_1, j_2]\cup\dots\cup[j_{2k-2}, j_{2k-1}]\cup[j_{2k-1}, q]$$ 
    where points $j_i$ are ordered linearly in $[p,q]$.
     Set $g(j_{2m-1})=e_m$, $g(j_{2m})=r$ for $m=1,\dots, k-1$ and $g(j_{2k-1})=e_k$. Then extend $g$ piecewise linearly, setting a $3$-fold between

    \begin{itemize}
        \item $[p, j_1]\textnormal{ and }[g(p), e_1]$,
        \item $[j_{2m-1}, j_{2m}]\textnormal{ and }[e_m, r],\; m=1,\dots,k-1,$
        \item $[j_{2m}, j_{2m+1}]\textnormal{ and }[r, e_{m+1}],\; m=1,\dots,k-1$,
        \item $[j_{2k-1}, q]\textnormal{ and }[e_{k}, g(q)]$.
    \end{itemize}

    \begin{remark}
        This is the first occurrence of a pattern used throughout the construction: after prescribing the images of finitely many marked points on a free arc, the map is extended linearly on the complementary subarcs. Note that this method of extending a map $g$ over an arc $[p, q]$ with fixed images $g(p), g(q)$ satisfies Definition \ref{piecewise_linearity_def} of piecewise linearity because every such arc $[p, q]$ is free by design. Moreover, whenever the image has finitely many linear components (e.g. in the case of a tree $X_{[p, q]}$), this holds true regardless of the choice of endpoints $j_1,\dots,j_k\subset(p, q)$. We will use this method of piecewise linear extension repeatedly in the later stages of the paper.
    \end{remark}

    Let $P([p, q])$ be the set of all endpoints of linearity intervals in $[p, q]$.
    
    \begin{claim}\label{g_slope_2}
        The slope $s_g>2$.
    \end{claim}
    \begin{proof}
        Observe that for any $P$-basic interval $I$ we have $\diam(I)<\delta/2$ by the choice of the initial $\delta/2$-dense set $P$. On the other hand, the diameter of its image $\diam(g(I))>\delta$. This follows from the fact that for every $[p, q]$ with $p, q$ - adjacent we chose $r\in P$ such that $\dist(r, [g(p), g(q)])>3\delta$. Thus, the image $g(I)$ must be an arc $[r, e]$ for some endpoint $e\in \End(X_{[p, q]})$, and we have that $\dist(e, [g(p), g(q)])<2\delta$.
    \end{proof}
    \begin{claim}\label{g_exact_markov}
        The map $g$ is Markov and exact.
    \end{claim}
    \begin{proof}
        Let us start by proving that $g$ is Markov. Let $P_\textnormal{adj}$ be the set of all pairs $(p, q)$ for $p, q\in P$ that are adjacent in $T_M$. Observe that each such pair of points forms an interval $[p, q]$ on which we defined a piecewise linear action of $g$. One can see that the set
        \begin{equation*}
            P_g=\bigcup_{(p, q)\in P_\textnormal{adj}} P([p, q])
        \end{equation*}
        is a finite $g$-invariant set such that for each connected component $C\subset T_M\setminus P_g$ the map $g$ is linear on $\overline{C}$  - as the map $g$ has been explicitly defined in this way.

        We will show that $g$ is transitive. Observe that the Markov graph $G_g$ is not a cycle, as there are at least two distinct outgoing edges from every vertex in $G_g$. Thus, it is sufficient to show that $G_g$ is strongly connected. Then the map $g$ must be transitive by Theorem \ref{connected_transitive}.
        
        Let $[j_1, j_2]$ be any $P_g$-basic interval. Then $g([j_1, j_2])$ must necessarily contain $[p, q]$ for some $(p, q)\in P_\textnormal{adj}$. Hence, there exists $k\in\{0,\dots,N-1\}$ such that  $f^{k}(o)\in g^{2}([j_1, j_2])$ as the image $g([p, q])$ contains $\bigcup_{x\in[p, q]} B(h(x), \delta)$. Now, by the choice of initial $\delta/2$-dense set $P$, any $P_g$-basic interval is contained in $B(f^{k}(o), \delta)$ for some $k\in\{0,\dots,N-1\}$. Thus, $G_g$ must be strongly connected by the choice of periodic point $o$, whose orbit is $\delta/2$-dense.

        To prove that $g$ is exact, it suffices to show that $g$ is totally transitive. We will do so by proving that there cannot be any finite family of nondegenerate subtrees with pairwise disjoint interiors satisfying Theorem \ref{totally_transitive_subtrees_Markov}. Assume that there is such a family $\mathcal{D}$ and thus, there are $X_1, X_2\in\mathcal{D}$ and let $e\in X_1\cap X_2$ be their common endpoint. It follows that necessarily $e\in P_g$. Then $g(e)=g(X_1)\cap g(X_2)$ is also an endpoint of two distinct trees from the family. For $[e, p]\subset X_1$, which is a $P_g$-basic interval for $g$ we have that $g([e, p])$ contains the set $\bigcup_{x\in[e, p]}B(h(x),\delta)$. Recall that $d_{\G}(g(x), h(x))<\delta/2$. Therefore, $g(X_1)=g(X_2)$ because they have intersection with a nonempty interior. We then obtain that $X_1=g^{n}(X_1)=g^{n}(X_2)=X_2$. This is a contradiction. Thus $g$ must be exact by Theorem \ref{totally_transitive_exact}.
    \end{proof}

    \begin{claim}\label{diam_g_pq}
        The distance $d_{\sup}(f, g)<9\epsilon$.
    \end{claim}
    \begin{proof}
        Fix any $x\in T_M$ and let $p, q\in P$ be adjacent points such that $x\in[p, q]$. Then $h(x)\in X_{[p, q]}$ and $g(x)\in B(X_{[p, q]}, 4\delta)$. Thus $$d(f(x), g(x))\leq \diam(X_{[p, q]})+4\delta = \epsilon+4\delta+4\delta<9\epsilon.$$
        Since we have that $d_{\sup} (f, h)<\epsilon$, the assertion follows.
    \end{proof}
    
    We will select a family of special $P_g$-basic intervals which will be important for the construction. Denote $\End(T_M)=\{e_1,\dots, e_{2^M}\}$. For each $n=1\dots,2^M$ we choose a $P_g$-basic interval $[p_n, q_n]$ such that $g(p_n)=e_n$. Then we set
    \begin{equation*}
        \mathcal{J}_\textnormal{exit}=\{[p_n, q_n]\colon\; n=1,\dots, 2^M\}.
    \end{equation*}
    We assume that the intervals $\mathcal{J}_\textnormal{exit}$ are pairwise disjoint, since one can always pick the middle subinterval in each $3$-fold.
   
    \begin{step}
        Extending $g$ to a map $\hat{f}_\textnormal{NME}\colon\G\to\G$.
    \end{step}
    
    We will define the map $\hat{f}_\textnormal{NME}\colon\G\to\G$ with respect to the tree $T_M$ and the associated map $g\colon T_M\to T_M$. Firstly, for all 
    $$x
    \in \overline{T_M\setminus \bigcup\mathcal{J}_\textnormal{exit}}=\overline{T_M\setminus \bigcup_{n=1}^{2^M}[p_n, q_n]}
    $$ 
    we define
    \begin{equation*}
        \hat{f}_\textnormal{NME}(x)=g(x).
    \end{equation*}

    We now turn to $\G\setminus T_M$. Observe that, since we have that $g(\End(T_M))\subset \End (T_M)$, there exists a periodic point for $g$ inside $\End(T_M)$.  Fix a periodic point $o'\in \End(T_M)$ and let $N'$ be its period. Let $I_k=[g^{k}(o'), r_k]$ be an edge in $\G$ for $k=0,\dots, N'-1$, where $r_k\in B(\G)$ is a branching point, whose binary code in $\G$ ends with $1$. In other words, $r_k$ is the right child of the vertex $g^{k}(o')$ (see Section \ref{the_family_of_trees}). Define $\mathcal{S}$ as a connected component of $\overline{\G\setminus \bigcup_{k=0}^{N'-1}[g^{k}(o'), r_k]}$ such that $T_M\subset \mathcal{S}$, which means that $\mathcal{S}$ is a dendrite obtained by removing first all ``rightmost" edges in the level $\mathcal{L}_{M+1}$ that have $g^{n}(o')$ for some $0\leq n< N'$ as their endpoint and then discarding from all the resulting connected components the ones that do not contain $T_M$ inside them. 
    We will denote by $A_k$ the connected component (Gehman dendrite) discarded at this step which has a common point with $I_k$ (its root). We will come back to these components at a much later step of the construction.
     
    Before going any further, observe that the dendrite $\mathcal{S}$ is also a Gehman dendrite. Indeed, one can see that $\mathcal{S}$ is formed by connecting a subfamily of the Gehman dendrites $\G_1,\G_2,\dots\subset \overline{\G\setminus T_{M+1}}$ (see Remark \ref{retraction_map}) with a finite tree $T\subset T_{M+1}$ in such a way that for every endpoint $e_n\in\End(T)$ there exists $\G_n$ such that the root of $\G_n$ and $e_n$ are identified in $\mathcal{S}$. Thus, all branching points are of order 3, the set of endpoints of $\mathcal{S}$ is a Cantor set and we invoke Theorem \ref{Gehman_unique}. We consider the same metric $d_{\G}$ on $\mathcal{S}$. Observe that by the choice of parameters $\delta', \epsilon'$ we have that for connected component $C\subset\mathcal{S}\setminus T_M$ its image $f(C)\subset \overline{\G\setminus T_{M_0}}$, since $f(\End(\G))\subset\End(\G)$. 
    We can present $\overline{C}=I_C\cup A_C$, where $I_C\subset \mathcal{L}_{M+1}$ is an edge and $A_C$ is a homeomorphic copy of $\G$.
\begin{figure}
\begin{tikzpicture}[scale=1]
	\tikzset{
		bp/.style={circle,fill=black,inner sep=1.2pt},
		rt/.style={circle,fill=white,draw=black,inner sep=1.4pt},
		removed/.style={dashed, red, thick},
		kept/.style={thick, blue},
		dendrite/.style={fill=blue!7, draw=blue!80!black, thick, rounded corners=2pt},
		discarded/.style={fill=red!7, draw=red!80!black, dashed, rounded corners=2pt}
	}
	\draw[fill=blue!4, draw=blue, thick] (0, 0) -- (-4.6, -3) -- (4.6, -3) -- cycle;
	\node[rt, label=above:{$b_\lambda$ (root)}] at (0,0) {};
	\node at (0, -1.5) {$T_M \subset \mathcal{S}$};

	\node[bp, label=above right:{$o'$}] (o) at (-3.2, -3) {};
	\node[bp, label=above left:{$g^j(o')$}] (go) at (3.2, -3) {};

	\node[bp, label=above:{$e_n\neq g^i(o')$}] (e) at (0, -3) {};
	\draw[kept] (e) -- (-0.8, -4);
	\draw[kept] (e) -- (0.8, -4);
	\draw[dendrite] (-0.8, -4) -- (-1.5, -5.5) -- (-0.1, -5.5) -- cycle;
	\draw[dendrite] (0.8, -4) -- (0.1, -5.5) -- (1.5, -5.5) -- cycle;
	\node at (0, -4.8) {$\dots$};
	
	\node[bp] (oL) at (-4.0, -4) {};
	\node[bp, label=right:{$r_0$}] (oR) at (-2.4, -4) {};
	\draw[kept] (o) -- (oL) node[midway, left=2pt] {\color{black} $I_C$};
	\draw[removed] (o) -- (oR) node[midway, right=2pt, text=red] {$I_0$};
	
	\node[bp] (goL) at (2.4, -4) {};
	\node[bp, label=right:{$r_j$}] (goR) at (4.0, -4) {};
	\draw[kept] (go) -- (goL);
	\draw[removed] (go) -- (goR);
	
	\draw[dendrite] (oL) -- (-4.7, -5.5) -- (-3.3, -5.5) -- cycle;
	\node at (-4.0, -4.5) {$A_C$};
	\node at (-4.0, -4.9) {$||$};
	\node at (-4.0, -5.2) {$\mathcal{G}_1\subset \mathcal{S}$};
	
	\draw[dendrite] (goL) -- (1.7, -5.5) -- (3.1, -5.5) -- cycle;
	\node at (2.4, -5.2) {$\mathcal{G}_j \subset \mathcal{S}$};
	
	\draw[discarded] (oR) -- (-3.1, -5.5) -- (-1.7, -5.5) -- cycle;
	\node[red!80!black] at (-2.4, -4.8) {$A_0$};
	\node[red!80!black] at (-2.4, -5.2) {\scriptsize(discarded)};
	
	\draw[discarded] (goR) -- (3.3, -5.5) -- (4.7, -5.5) -- cycle;
	\node[red!80!black] at (4.0, -4.8) {$A_j$};
	\node[red!80!black] at (4.0, -5.2) {\scriptsize(discarded)};
	
	\draw[decorate, decoration={brace, amplitude=10pt, mirror}, thick, blue!80!black] (-4.8, -5.8) -- (3.2, -5.8) node[midway, below=12pt] {Gehman Dendrite $\mathcal{S}$};
	
\end{tikzpicture}
\caption{Schematic structure of dendrite $\mathcal{S}$}\label{fig:ConstructionS}
\end{figure}
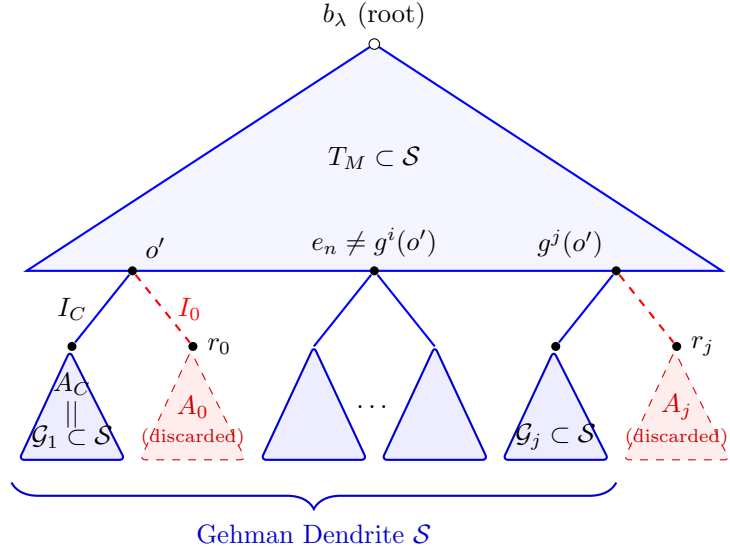
    Fix the closure $\overline{C}=I_C\cup A_C$ of a connected component $C\subset\mathcal{S}\setminus T_M$. There exists a connected component $A'_C\subset \overline{\mathcal{S}\setminus T_{M_0}}$ with $f(A_C)\subset A'_C$ and $A_C, A'_C$ are homeomorphic (since each of them is homeomorphic to the Gehman dendrite). Let $r(A_C)=A_C\cap I_C, r(A'_C)=A'_C\cap T_{M_0}$ be roots of these dendrites. There exists a natural homeomorphism $h_{f(A)}\colon A_C\to A'_C$ such that $h(r(A_C))=r(A'_C)$. For each element $x$ of the countable set $x\in \End(A_C)\cup B(A_C)\cup \{r({A_C})\}$ we define $$\hat{f}_\textnormal{NME}(x)=h_{f(A)}(x).$$

    \begin{remark}\label{skeleton}
        Recall that the original map $f\colon \G\to\G$ is surjective. Thus, for each connected component $D\subset\overline{\mathcal{S}\setminus T_{M_0}}$ we may find at least one $C$ and $A_C\subset\mathcal{S}\setminus T_M$ such that $f(A_C)\subset A_{C}'=D$. Since we defined the map $\hat{f}_\textnormal{NME}$ on $\End(A_C)\cup B(A_C)\cup \{r({A_C})\}$ in such a way that all these special points in $A_C$ are mapped onto their respective copies via a homeomorphism, we have uniquely transferred all of the $A_C$'s structure onto $A_C'$. Thus, the skeleton $\End(A_C)\cup B(A_C)\cup \{r({A_C})\}$ of the dendrite $A_C$ was fully mapped onto the skeleton $\End(A_C')\cup B(A_C')\cup \{r({A_C'})\}$  of $A_C'$ and for any component $A_C'$ there's at least one $A_C$, whose skeleton has been mapped onto $A_C'$.
    \end{remark}

    For each of the endpoints of $I_C=[e, r(A_C)]$, where $e\in \End(T_M)$, the images $\hat{f}_\textnormal{NME}(e)$ and $\hat{f}_\textnormal{NME}(r(A_C))$ have already been defined. We define $\hat{f}_\textnormal{NME}$ on $[e, r(A_C)]$ as the linear stretch over the arc $[\hat{f}_\textnormal{NME}(e), r(A'_C)]$. 

Since $\mathcal{S}$ is the Gehman dendrite, there exists an enumeration system for all its branching points in a manner similar to that presented for $\G$ in Section \ref{the_family_of_trees}. Moreover, one can see that the only branching points in $\G$ that are not branching points in $\mathcal{S}$ are the ones forming the periodic orbit $g^{k}(o')$. Thus, for all components $\overline{C}=I_C\cup A_C$, where the edge $I_C=[g^k(o'), b]$ is for $b\in B(\mathcal{S})$, the enumeration of the branching points in the dendrite $A_C$ in $\mathcal{S}$ is different from the one in $\G$, and this occurs only for those dendrites. We extend the notion of the $n$-th level $\mathcal{L}_n(\mathcal{S})$ in $\mathcal{S}$ to be the union of all edges $[b_{\omega}, b_{\omega'}]\subset\mathcal{S}$ such that the enumerating binary sequence $\omega'\in\{0, 1\}^n$. Note that there are edges in $\mathcal{S}$ that are formed by a union of two edges in $\mathcal{G}$ - and those are precisely the edges that contain $g^k(o')$ for some $k\in\{0,\dots, N'-1\}$. Let $\mathcal{F}_n=\overline{\bigcup_{m=M+(n-1)d+1}^{M+nd}\mathcal{L}_m(\mathcal{S})}$ for $n\in\mathbb{N}$. We will call $\mathcal{F}_n$ the \emph{$n$-th floor} in $\mathcal{S}$. Let $B_n(\mathcal{S})=\mathcal{F}_{n-1}\cap\mathcal{F}_n$ for $n\in\mathbb{N}_1$ and $B_0(\mathcal{S})=\End(T_{M_0})$.

    \begin{claim}\label{level_up}
        Let $[b, b']\subset\mathcal{F}_n$  for $n\in\mathbb{N}_1$ be a free arc. Then the images $\hat{f}_\textnormal{NME}(b)$, $\hat{f}_\textnormal{NME}(b')\in\mathcal{F}_{n-1}$ and $d_{\G}(\hat{f}_\textnormal{NME}(b), \hat{f}_\textnormal{NME}(b'))\geq 4d_{\G}(b, b')$.
    \end{claim}
    \begin{proof}
        The first part of the claim follows from the definition of floors in $\mathcal{S}$ and the fact that we defined $\hat{f}$ on $B(\mathcal{S})$ using natural homeomorphisms identifying the roots of respective subdendrites in $\mathcal{S}$.
        
        Assume that $[b, b']\subset \overline{C}$, where $\overline{C}=I_C\cup A_C$ and $C\subset \mathcal{S}\setminus T_M$ is a connected component. The assertion is immediate if $[b, b']\subset I_C$. Assume that $[b, b']\subset A_C$. Since the image $\hat{f}_\textnormal{NME}(r(A))=r(A')$ for appropriately chosen component $A'\subset\mathcal{S}\setminus T_{M_0}$, it follows that $\hat{f}(r(A))\in \mathcal{L}_{M_0-1}(\mathcal{S})\cap\mathcal{L}_{M_0}(\mathcal{S})$, i.e. the root has been moved $d$ levels up in $\mathcal{G}$. Now observe that by the discrepancy of the branching point enumeration between the dendrites $\mathcal{S}$ and $\G$, there are branching points that have been moved up $d-1$ levels in $\mathcal{G}$. This follows from the fact that there are edges in $\mathcal{S}$ which are unions of two edges in $\mathcal{G}$ and hence, a branching point might have been moved two levels below its parent in $\G$ to remain a branching point in $\mathcal{S}$. Moreover, if a branching point has been moved $d-1$ levels up in $\G$, then its child has been moved $d-1$ levels up in $\G$.

        If both $b, b'$ have been moved $d$ levels up in $\G$, then $d_{\G}(\hat{f}_\textnormal{NME}(b), \hat{f}_\textnormal{NME}(b'))=4^d d_{\G}(b, b')$. If $b$ has been moved $d$ levels up in $\G$, but $b'$ has been moved $d-1$ levels up in $\G$, then $d_{\G}(\hat{f}_\textnormal{NME}(b), \hat{f}_\textnormal{NME}(b'))\geq (4^d+4^{d-1})d_{\G}(b, b')$. If both $b, b'$ have been moved $d-1$ levels up in $\G$, then $d_{\G}(\hat{f}_\textnormal{NME}(b), \hat{f}_\textnormal{NME}(b'))\geq 4^{d-1}d_{\G}(b, b')$. Since $d\geq 2$, the assertion follows.
    \end{proof}

    Fix $[b, b']\subset\mathcal{F}_n$ for some $n\in\mathbb{N}_1$ which is an edge in $\G$. There exists a tree $X_{3d}([b, b'])$ - a connected component of $\mathcal{F}_{n-1}\cup\mathcal{F}_n\cup\mathcal{F}_{n+1}$ such that the images $\hat{f}_\textnormal{NME}(b), \hat{f}_\textnormal{NME}(b')\in X_{3d}([b, b'])$. One can see that $X_{3d}$ is a full binary tree of height $3d$. Denote the set of its endpoints $\End(X_{3d}([b, b']))=\{e_1,\dots,e_{2^{3d}}\}$ and let $r_{[b, b']}\in B_{n-1}(\mathcal{S})$ be its root. Assume that $e_{2^{3d}}$ is an endpoint such that the interval $[b', e_{2^{3d}}]$ contains $r_{[b, b']}$. Divide $[b, b']$ into $2^{3d+1}$ subintervals of equal length with pairwise disjoint interiors, i.e. 
    $$[b, b']=[b, d_1]\cup[d_1, u_1]\cup[u_1, d_2]\cup\dots\cup[u_{2^{3d}-1}, d_{2^{3d}}]\cup[d_{2^{3d}}, b'].$$
    
    Let $P([b, b'])= \{b, d_1, u_1,\dots,d_{2^{3d}-1}, u_{2^{3d}-1}, d_{2^{3d}}, b'\}\subset [b, b']$ be the set of all endpoints of these subintervals. Define $\hat{f}_\textnormal{NME}(u_1)=\dots=\hat{f}_\textnormal{NME}(u_{2^{3d}-1})=r_{[b, b']}$ and $\hat{f}_\textnormal{NME}(d_k)=e_k$ for $k=1,\dots,2^{3d}$. Then extend $\hat{f}_\textnormal{NME}$ linearly on each of the subintervals in $[b, b']$. We obtain
    \begin{equation*}
        \left\{
        \begin{split}
            &\hat{f}_\textnormal{NME}([b, d_1])=[\hat{f}_\textnormal{NME}(b), e_1],\\
            &\hat{f}_\textnormal{NME}([d_k, u_k])=[r_{[b, b']}, e_k],\;k=1,\dots, 2^{3d}-1,\\
            &\hat{f}_\textnormal{NME}([u_{k}, d_{k+1}])=[r_{[b, b']}, e_{k+1}]\;k=1,\dots, 2^{3d}-1\\
            &\hat{f}_\textnormal{NME}([d_{2^{3d}}, b'])=[\hat{f}_\textnormal{NME}(b'), e_{2^{3d}}].
        \end{split}
        \right.
    \end{equation*}
    
    We will now define $\hat{f}_\textnormal{NME}$ on the family $\mathcal{J}_\textnormal{exit}$. By the definition, this family is composed of $2^M$ pairwise disjoint intervals $[p_n, q_n]$ with $g(p_n)=\hat{f}_\textnormal{NME}(p_n)\in\End(T_M)$. Moreover, the map $\hat{f}_\textnormal{NME}$ has already been defined on all $p_n, q_n$ for $n=1,\dots,2^M$. Fix an interval $[p_n, q_n]$ and let $X_{d}(g(p_n))\subset \mathcal{F}_1\cup \bigcup_{k=0}^{N'-1}[g^{k}(o'), r_k]$ denote the connected component such that $g(p_n)$ is its root, where 
    branching points $r_k$  were used to define $\mathcal{S}$.
    Denote $\End(X_{d}(g(p_n)))=\{e_1,\dots,e_K\}$. Divide $[p_n, q_n]$ into $2K+1$ subintervals of equal length with pairwise disjoint interiors, i.e.
    \begin{equation*}
        [p_n, q_n]=[p_n, u_1]\cup[u_1, d_1]\cup[d_1, u_2]\cup\dots\cup[u_K, d_K]\cup[d_K, q_n].
    \end{equation*}
    Let $P([p_n, q_n])=\{p_n, u_1, d_1,\dots,u_{K}, d_{K}, q_n\}\subset [p_n, q_n]$ be the set of all endpoints of these subintervals. Set $\hat{f}_\textnormal{NME}(u_1)=\dots=\hat{f}_\textnormal{NME}(u_{K})=g(q_n)$ and $\hat{f}_\textnormal{NME}(d_k)=e_k$ for $k=1,\dots,K$. Then extend $\hat{f}_\textnormal{NME}$ linearly on each of the subintervals in $[p_n, q_n]$. We obtain
    \begin{equation*}
        \left\{
        \begin{split}
            &\hat{f}_\textnormal{NME}([p_n, u_1])=[g(p_n), g(q_n)],\\
            &\hat{f}_\textnormal{NME}([u_k, d_k])=[g(q_n), e_k],\;k=1,\dots, K,\\
            &\hat{f}_\textnormal{NME}([d_{k}, u_{k+1}])=[g(q_n), e_{k}],\;k=1,\dots, K-1\\
            &\hat{f}_\textnormal{NME}([d_{K}, q_n])=[e_{K}, g(q_n)].
        \end{split}
        \right.
    \end{equation*}

    Note that for any $J\in\mathcal{J}_\textnormal{exit}$ we have $\hat{f}_\textnormal{NME}(J)\cap \G\setminus\mathcal{S}\neq\emptyset$, which will be crucial for the further steps of the construction. At this point the map $\hat{f}_\textnormal{NME}$ has been completely defined on $\mathcal{S}$.

    \begin{claim}\label{pietra_surjektywne}
        We have $\mathcal{S}\subset \hat{f}_\textnormal{NME}(\mathcal{S})$.
    \end{claim}
    \begin{proof}
        The following claim follows from Remark \ref{skeleton} and the definition of $g$. One can see that a straightforward consequence of the aforementioned observation is that $\hat{f}_\textnormal{NME}(\mathcal{F}_n)=\mathcal{F}_{n-1}\cup\mathcal{F}_n\cup\mathcal{F}_{n+1}$ for all $n\in\mathbb{N}_1$. Thus, the assertion follows.
    \end{proof}

    \begin{claim}\label{L-lipschitz_constants}
        The map $\hat{f}_\textnormal{NME}|_{\mathcal{S}}$ is $\sigma$-linear with the slope bounded from above on $\mathcal{S}$, that is $$\sup\{s_{\hat{f}_\textnormal{NME}}(I)\colon I\subset\mathcal{S}\textnormal{ - linearity interval}\}<\infty.$$
    \end{claim}
    \begin{proof}
        One can see that the map $\hat{f}_\textnormal{NME}|_{\mathcal{S}}$ is $\sigma$-linear by the definition. Let $B_\textnormal{adj}(\mathcal{S})$ be the set of all pairs $(b, b')$ for adjacent $b, b'\subset B(\mathcal{S})$. Then the set
        \begin{equation*}
            P'=P_g\cup \bigcup_{(b, b')\in B_\textnormal{adj}(\mathcal{S})} P([b, b'])\cup\bigcup_{n=1}^{2^M} P([p_n, q_n])
        \end{equation*}
        defines a countable dense union of free arcs in $\mathcal{S}$ such that $\hat{f}_\textnormal{NME}$ is linear on each of those arcs. Note that $B(\mathcal{S})\subset P_g\cup \bigcup_{(b, b')\in B_\textnormal{adj}(\mathcal{S})} P([b, b'])$ and also $\End(T_M)\subset P_g$, hence  $B(\G)\cap\mathcal{S}\subset P'$.
        
        Since $\hat{f}_\textnormal{NME}(T_{M})\subset T_{M+1}\cup \mathcal{F}_1$, the slope on $T_M$ is bounded by some positive constant $L_1=\diam(T_{M+1}\cup \mathcal{F}_1)$. Let $[b, b']\subset\mathcal{F}_{n}$ be an edge  for $n\in\mathbb{N}_1$. It follows that if $\diam([b, b'])=d_n$, then $\diam{(X_{3d}([b, b']))}\leq 2\sum_{k=0}^{3d}4^kd_n=\frac{2}{3}(4^{3d+1}-1)d_n$, where in the last inequality we take into account that in level $\mathcal{L}_{M}(\mathcal{S})$ there are edges that  are unions of edges from $\mathcal{L}_M\cup\mathcal{L}_{M+1}$ in $\mathcal{G}$. Moreover, in each $[b, b']$ there are exactly $2^{3d+1}$ linearity intervals of equal length. Thus, the slope on $\mathcal{S}\setminus T_M$ is bounded by $L_2=\frac{2^{3d+2}}{3}(4^{3d+1}-1)$. Thus, the slope $s_{\hat{f}_\textnormal{NME}}(I)<L=\max\{L_1, L_2\}$ for any linearity interval $I$ in $\mathcal{S}$.
    \end{proof}
    \begin{claim}\label{L-lipschitz_argument}
        The map $\hat{f}_\textnormal{NME}|_{\mathcal{S}}$ is $L$-Lipschitz continuous.
    \end{claim}
    \begin{proof}
        The assertion follows from the definition of the metric $d_{\G}$. Fix two points $x, y\in\mathcal{S}\setminus\End(\mathcal{S})$. The arc $[x, y]\subset\bigcup_{k=1}^n [p_{k}, p_{k+1}]$ for some $p_1,\dots,p_{n+1}\in P'$ with $p_k, p_{k+1}$ adjacent. Thus, since $$d_{\G}(x, y)=d_{\G}(x, p_1)+d_{\G}(p_1, p_2)+\dots d_{\G}(p_{n-1}, p_{n})+d_{\G}(p_{n}, y)=d(x, y)$$
        we have that
        \begin{equation*}
            \begin{split}
                d_{\G}(\hat{f}_\textnormal{NME}(x), \hat{f}_\textnormal{NME}(y))&=d_{\G}(\hat{f}_\textnormal{NME}(x), \hat{f}_\textnormal{NME}(p_1))+\dots+d_{\G}(\hat{f}_\textnormal{NME}(p_{n+1}), \hat{f}_\textnormal{NME}(y))\\
                &< L\cdot d_{\G}(x, y),
            \end{split}
        \end{equation*}
        where $L$ is an upper bound of the slope $s_{\hat{f}_\textnormal{NME}|_{\mathcal{S}}}$ of the $\sigma$-linear map $\hat{f}_\textnormal{NME}$. Now if $x, y\in\End(\mathcal{S})$, then there exist $(b_n)_{n=1}^\infty, (b'_n)_{n=1}^\infty\subset B(\mathcal{S})$ be such that $$\lim_{n\to\infty}b_n=x, \;\lim_{n\to\infty}b'_n=y.$$ One can see that for all $n\in\mathbb{N}$ we have
        \begin{equation*}
            \begin{split}
                d_{\G}(\hat{f}_\textnormal{NME}(x), \hat{f}_\textnormal{NME}(y))\leq& d_{\G}(\hat{f}_\textnormal{NME}(x), \hat{f}_\textnormal{NME}(b_n))
                +d_{\G}(\hat{f}_\textnormal{NME}(b'_n), \hat{f}_\textnormal{NME}(y))\\
                &\quad +L\cdot d_{\G}(b_n, b'_n)
            \end{split}
        \end{equation*}
        and hence $d_{\G}(\hat{f}_\textnormal{NME}(x), \hat{f}_\textnormal{NME}(y))\leq L \cdot d_{\G}(x, y)$.
    \end{proof}
    
    Recall that by Remark~\ref{rem:linearity}, linearity interval is maximal in the sense of inclusion interval on which the map is linear.
    
    \begin{claim}\label{claim_8}
        Each linearity interval $\hat{f}_\textnormal{NME}|_\mathcal{S}$-covers finitely many linearity intervals in $\mathcal{S}$ and for each linearity interval the number of $\hat{f}_\textnormal{NME}|_\mathcal{S}$-covered linearity intervals is bounded by some universal constant $S$.
    \end{claim}
    \begin{proof}
        Since $\hat{f}_\textnormal{NME}(T_M)\subset T_M\cup \mathcal{F}_1$, each linearity interval in $T_M$ can cover at most $S_1$ linearity intervals, where $S_1<\infty$ is the number of linearity intervals in $T_M\cup \mathcal{F}_1$. The same holds true for linearity intervals $I_C$, where $C\subset \mathcal{S}\setminus T_M$ is a connected component. Now fix any linearity interval in  any floor $\mathcal{F}_n$ for $n\in\mathbb{N}_1$. The image of each edge in $\mathcal{F}_n$ is contained in a connected component of $\mathcal{F}_{n-1}\cup\mathcal{F}_n\cup\mathcal{F}_{n+1}$,
        which has at most $2^{3d}$ edges. By the construction, each edge in $\mathcal{F}_n$ has at most $2^{3d+1}$ linearity intervals,
        so each such interval covers at most $S_2=4^{3d+1}$ linearity intervals.
        We put $S=\max\{S_1, S_2\}$ to complete the proof.
    \end{proof}

    We will now define $\hat{f}_\textnormal{NME}$ on $\G\setminus\mathcal{S}$.
    Let us start by defining $\hat{f}_\textnormal{NME}$ on $(B(\G)\cup\End(\G))\setminus\mathcal{S}$. One can see that there are exactly $N'$ components $I_0\cup A_0,\dots,I_{N'-1}\cup A_{N'-1}\subset \overline{\G\setminus \mathcal{S}}$, where $N'$ is the period of  the periodic point $o'$ (see Figure~\ref{fig:ConstructionS}). Recall that $I_k=[g^{k}(o'), r_k]$, and $r_k$ is the root of $A_k$. We set $$\hat{f}_\textnormal{NME}(r_k)=r_{k+1\;\textrm{mod } N'}.$$ 

    Let $\Tilde{\G}=\{0,\dots,N'-1\}\times\G$ be a product of $N'$ copies of the Gehman dendrite. We define a skew product $F$ on  $\Tilde{\G}$ by
    \begin{equation*}
        F(k, b_\omega)=\left\{\begin{split}
            &(k+1, b_\omega),\; k=0,\dots, N'-2\\
            &(0, b_{\omega\oplus 1}),\; k=N'-1,
        \end{split}\right.
    \end{equation*}
    where $k=0,\dots N'-1$ and $b_\omega\in B(\G)\cup\{b_\lambda\}$. Here, we define $\omega\oplus 1$ to be the binary expansion of the integer $\langle\omega\rangle+1\;(\textrm{mod } 2^n)$, where, for a binary word $\omega=\omega_0\ldots \omega_{n-1}\in\{0,1\}^n$, we interpret it as a binary expansion of the integer $\langle \omega\rangle=\omega_0 2^0+\omega_1 2^1+\ldots+\omega_{n-1}2^{n-1}$ with the least significant digit written first, and we always use $n$ digits. Moreover, we assume that $b_{\lambda\oplus 1}=b_\lambda$, i.e., the root of $\G$ is invariant under the $\oplus$-addition operator. One can see that the $\oplus$ operator is defined in such a way that if $[b_\omega, b_{\omega'}]$ is an edge, then $[b_{\omega\oplus 1}, b_{\omega'\oplus 1}]$ is also an edge. We then define $$F(k, [b_\omega, b_{\omega'}])=(k+1 \;\textrm{mod } N', [F(k, b_\omega), F(k, b_{\omega'})]).$$ By the choice of metric $d_{\G}$ this map is an isometry on $\{0,\dots,N'-1\}\times\G\setminus\End(\G)$. Hence, for any $(k,e)\in \{k\}\times\End(\G)$ the following limit exists
    \begin{equation*}
        F(k, e)=\lim_{n\to\infty}F(k, b_n),
    \end{equation*}
    where $(b_n)_{n=1}^\infty\subset B(\G)$ is a Cauchy sequence such that $\lim_{n\to\infty}b_n=e$. As a byproduct of this definition we immediately obtain that $F$ is continuous.

    For fixed $k$ let $h_{A_k}\colon\{k\}\times\G\to A_k$ be a natural homeomorphism with $h_{A_k}(k, b_\lambda)=r_{k}$. For each $k$ and for $b\in \{r_k\}\cup B(A_k)\cup\End(A_k)$ we define $$\hat{f}_\textnormal{NME}(b)=h_{A_{k+1\;\textrm{mod } N'}}\circ F\circ h_{A_k}^{-1}(k, b).$$
    Observe that if we set $b_\omega(k)=h_{A_k}(k, b_\omega)$, then the map $\hat{f}_\textnormal{NME}$ acts in the following way:
    \begin{equation*}
        \hat{f}_\textnormal{NME}(b_\omega(k))=\left\{
        \begin{split}
            &b_{\omega}(k+1),\; k=0,\dots N'-2,\\
            &b_{\omega\oplus 1}(0),\; k=N'-1.
        \end{split}\right.
    \end{equation*}

    Before we proceed with the next part of our construction, let us introduce the following notation: let $1^{(n)}=1\dots 1\in\{0, 1\}^n$ be the sequence of $n$ consecutive ones and similarly $0^{(n)}\in\{0, 1\}^n$ be the sequence of $n$ consecutive zeroes for $n\in\mathbb{N}$ and  in particular $1^{(0)}=0^{(0)}=\lambda$ is the empty word.

    Fix an edge $[b_\omega(k), b_{\omega'}(k)]\subset \overline{\G\setminus \mathcal{S}}$, where $\omega\in\{0, 1\}^n$ for some $n\in\mathbb{N}$ and $k\in\{0,\dots, N'-1\}$. Assume first that $k=N'-1$ and $\omega'$ is a sequence of ones $\omega'=1^{(n+1)}$. Observe that then $\omega'\oplus 1=0^{(n+1)}$. 

     Divide $[b_{1^{(n)}}(N'-1), b_{1^{(n+1)}}(N'-1)]$ into three subintervals of equal length with pairwise disjoint interiors, i.e. $$[b_{1^{(n)}}(N'-1), b_{1^{(n+1)}}(N'-1)]=[b_{1^{(n)}}(N'-1), j_1]\cup[j_1, j_2]\cup [j_2, b_{1^{(n+1)}}(N'-1)].$$ Let
    \begin{equation*}
        P([b_{1^{(n)}}(N'-1), b_{1^{(n+1)}}(N'-1)])=\{b_{1^{(n)}}(N'-1), j_1, j_2, b_{1^{(n+1)}}(N'-1)\}
    \end{equation*}
    be the set of all endpoints of these subintervals.
    First, let $n=0$ and set $$\hat{f}_\textnormal{NME}(j_1)=b_{0^{(2)}}(0), \hat{f}_\textnormal{NME}(j_2)=o'.$$ Then extend $\hat{f}_\textnormal{NME}$ linearly on each of the three subintervals inside $[b_{1^{(n)}}(N'-1), b_{1^{(n+1)}}(N'-1)]$. We obtain
     \begin{equation*}
        \left\{
        \begin{split}
            &\hat{f}_\textnormal{NME}([b_{\lambda}(N'-1), j_1])=[b_{\lambda}(0), b_{0^{(2)}}(0)],\\
            &\hat{f}_\textnormal{NME}([j_1, j_2])=[o', b_{0^{(2)}}(0)],\\
            &\hat{f}_\textnormal{NME}([j_2, b_{1}(N'-1)])=[o', b_{0}(0)].
        \end{split}
        \right.
    \end{equation*}

   Next, for $n=1,2, \ldots$, we repeat the same procedure, but with assigning the images
    \begin{equation*}
            \hat{f}_\textnormal{NME}(j_1)=b_{0^{(n+2)}}(0),\; \hat{f}_\textnormal{NME}(j_2)=b_{0^{(n-1)}}(0)
    \end{equation*}
    to the endpoints $j_1, j_2\in [b_{1^{(n)}}(N'-1), b_{1^{(n+1)}}(N'-1)]$. Then we extend $\hat{f}_\textnormal{NME}$ linearly on each of the three subintervals, obtaining:
    \begin{equation*}
        \left\{
        \begin{split}
            &\hat{f}_\textnormal{NME}([b_{1^{(n)}}(N'-1), j_1])=[b_{0^{(n)}}(0), b_{0^{(n+2)}}(0)],\\
            &\hat{f}_\textnormal{NME}([j_1, j_2])=[b_{0^{(n-1)}}(0), b_{0^{(n+2)}}(0)],\\
            &\hat{f}_\textnormal{NME}([j_2, b_{1}(N'-1)])=[b_{0^{(n-1)}}(0), b_{0^{(n+1)}}(0)].
        \end{split}
        \right.
    \end{equation*}
    
Next, we define the image $\hat{f}_\textnormal{NME}$ on $[b_{0^{(n)}}(0), b_{0^{(n+1)}}(0)]$ by setting a $3$-fold between the edges $[b_{0^{(n)}}(0), b_{0^{(n+1)}}(0)]$, $[\hat{f}_\textnormal{NME}(b_{0^{(n)}}(0)), \hat{f}_\textnormal{NME}(b_{0^{(n+1)}}(0))]$. Now assume that either $0<k< N'-1$ or $\omega'\neq 0^{(n+1)}, 1^{(n+1)}$ and let $S'$ be an odd integer such that 
$$S+3<S',$$ 
where the number $S$ is derived from Claim \ref{claim_8}. We define $\hat{f}_\textnormal{NME}$ on $[b_{\omega}(k), b_{\omega'}(k)]$ by setting an $S'$-fold between the edges $[b_{\omega}(k), b_{\omega'}(k)]$ and $[\hat{f}_\textnormal{NME}(b_{\omega}(k)), \hat{f}_\textnormal{NME}(b_{\omega'}(k))]$. 
Following previous approach, for each $[b_\omega(k), b_{\omega'}(k)]$ we define $P([b_\omega(k), b_{\omega'}(k)])$ to be the set of endpoints of linearity intervals in $[b_\omega(k), b_{\omega'}(k)]$.

     Lastly, we will define $\hat{f}_\textnormal{NME}$ on intervals $I_k$. Fix any $k=0,\dots, N'-1$ and let $e\in \End(T_{M_0})$ be such that $[e, g^{k+1}(o')]\subset\overline{\G\setminus T_{M_0}}$. As before, divide $I_k$ into three subintervals of equal length with pairwise disjoint interiors, i.e. 
    $$I_k=[g^{k}(o'), j_1]\cup[j_1, j_2]\cup [j_2, r_k].$$
    
    Let $P(I_k)=\{g^{k}(o'), j_1, j_2, r_k\}$ be the set of endpoints of these subintervals and define
    \begin{equation*}
    \hat{f}_\textnormal{NME}(j_1)=b_{0^{(1)}}(k+1\;\textrm{mod } N'), \;\hat{f}_\textnormal{NME}(j_2)=e.
    \end{equation*}
    To finalize the construction, extend $\hat{f}_\textnormal{NME}$ linearly on each subinterval of $I_k$, obtaining:
    \begin{equation*}
        \left\{
        \begin{split}
            &\hat{f}_\textnormal{NME}([g^{k}(o'), j_1])=[g^{k+1}(o'), b_{0^{(1)}}(k+1\;\textrm{mod } N')],\\
            &\hat{f}_\textnormal{NME}([j_1, j_2])=[e, b_{0^{(1)}}(k+1\;\textrm{mod } N')],\\
            &\hat{f}_\textnormal{NME}([j_2, r_k])=[e, r_{k+1\;\textrm{mod } N'}].
        \end{split}
        \right.
    \end{equation*}
    
    The construction of $\hat{f}_\textnormal{NME}$ has been completed.

    \begin{claim}
        The map $\hat{f}_\textnormal{NME}$ is a continous, $\sigma$-linear surjection. Hence, $(\G, \hat{f}_\textnormal{NME})$ is a TDS.
    \end{claim}
    \begin{proof}
        The above claim follows from reasoning similar to that presented in Claims \ref{pietra_surjektywne}, \ref{L-lipschitz_constants}, \ref{L-lipschitz_argument}.
    \end{proof}
    
    \begin{claim}\label{f_hat_markovian}
        The map $\hat{f}_\textnormal{NME}$ is Markovian.
    \end{claim}
    \begin{proof}
       We will prove it in a similar way as we did for the map $g\colon T_M\to T_M$. Let 
        \begin{equation*}
             P_{\hat{f}_\textnormal{NME}}=P'\cup \bigcup_{b, b'\in B_\textnormal{adj}(\G)} P([b, b']),
        \end{equation*}
        where $B_\textnormal{adj}(\G)$ is the set of all pairs of adjacent points $b, b'\in B(\G)$. 
         We will show that $\hat{f}_\textnormal{NME}$ is $P_{\hat{f}_\textnormal{NME}}$-Markovian.
        
          Note that $P_{\hat{f}_\textnormal{NME}}\cap\mathcal{S}=P'$ and $P'$ contain all sets $P([b_\omega(k), b_{\omega'}(k)])$ for $\omega'\in\{0, 1\}^{n+1}$.
        
        One can see that the set $P_{\hat{f}_\textnormal{NME}}$ is countable, invariant and does not contain any endpoint of $\G$. Indeed, we have that $\hat{f}_\textnormal{NME}(P_g\cup B(\G))\subset P_g\cup B(\G)$ and moreover $\hat{f}_\textnormal{NME}(P([b, b']))\subset B(\G)$, $\hat{f}_\textnormal{NME}(P([p_n, q_n]))\subset B(\G)$ for any $(b, b')\in B_\textnormal{adj}$ and $[p_n, q_n]\in\mathcal{J}_\textnormal{exit}$. The map $\hat{f}_\textnormal{NME}$ is clearly linear on each $\overline{C}$ for $C\subset\G\setminus (\End(\G)\cup P_{\hat{f}_\textnormal{NME}})$ (see also Claim~\ref{L-lipschitz_constants}). Moreover, we have that $P_{\hat{f}_\textnormal{NME}}\cap[b, b']$ is finite for every pair $(b, b')\in B_\textnormal{adj}$. 
       Connected components of $\G\setminus (\End(\G)\cup P_{\hat{f}_\textnormal{NME}})$ are  free arcs covering $\G\setminus \End(\G)$ and by definition 
        $\hat{f}_\textnormal{NME}(\End(\G))=\End(\G)$.
        Thus, ${\hat{f}_\textnormal{NME}}$ is Markovian with respect to $P_{\hat{f}_\textnormal{NME}}$.
    \end{proof}

     \begin{claim}\label{slope_hat_f_2}
        The slope $s_{\hat{f}_\textnormal{NME}}$ is bounded from below by $c>2$.
    \end{claim}
    \begin{proof}
        Assume first that $[p, q]$ is a $P_{\hat{f}_\textnormal{NME}}$-basic interval in $T_M$ and $[p, q]\notin \mathcal{J}_\textnormal{exit}$. Then $[p, q]$ is a $P_g$-basic interval, since in the construction of $\hat{f}_\textnormal{NME}$ we did not alter the action on such intervals. Thus, $s_{\hat{f}_\textnormal{NME}}([p, q])>2$ by Claim \ref{g_slope_2}.
        
        Assume now that $[p, q]\subset[p_n, q_n]\in\mathcal{J}_\textnormal{exit}$ is a $P_{\hat{f}_\textnormal{NME}}$-basic interval. One can see that $g([p, q])\subset g([p_n, q_n])\subset\hat{f}_\textnormal{NME}([p, q])$. Hence, $s_{\hat{f}_\textnormal{NME}}([p, q])>2$.
        
        Assume now that $[p, q]\subset [b_{1^{(n)}}(N'-1), b_{1^{(n+1)}}(N'-1)]$ for some $n\in\mathbb{N}$. It follows that $[b_{0^{(n)}}(0), b_{0^{(n+1)}}(0)]\subset \hat{f}_\textnormal{NME}([p, q])$ and $\diam([b_{0^{(n)}}(0), b_{0^{(n+1)}}(0)])=3\diam([p, q])$, as in the construction we've divided $[b_{1^{(n)}}(N'-1), b_{1^{(n+1)}}(N'-1)]$ into three subintervals of equal length. Thus, $s_{\hat{f}_\textnormal{NME}}([p, q])\geq 3$. The same reasoning applies to $P_{\hat{f}_\textnormal{NME}}$-basic intervals which are contained in the edge $[b_{0^{(n)}}(0), b_{0^{(n+1)}}(0)]$, as $\hat{f}_\textnormal{NME}$ acts on $[b_{0^{(n)}}(0), b_{0^{(n+1)}}(0)]$ as a $3$-fold.

        Assume now that $[p, q]\subset I_k$ for $k=0,\dots, N'-1$. Recall that $I_{k+1\;\textrm{mod } N'}\subset \hat{f}_\textnormal{NME}([p, q])$ and $\diam(I_{k+1\;\textrm{mod } N'})=\diam(I_k)$. Since $3\diam([p, q])=\diam(I_k)$ by construction, it follows that $s_{\hat{f}_{\textnormal{NME}}}([p, q])\geq 3$.

        Assume now that $[p, q]\subset\G\setminus\mathcal{S}$ is a $P_{\hat{f}_\textnormal{NME}}$-basic interval that has not been considered before. One can see that the map $\hat{f}_{\textnormal{NME}}$ acts on $[p, q]$ as an $S'$-fold and $S'>3$. Thus, the slope $s_{\hat{f}_\textnormal{NME}}([p, q])>3$.

        Assume now that $[p, q]\subset\mathcal{S}\setminus T_M$ is a $P_{\hat{f}_\textnormal{NME}}$-basic interval with $[p, q]\subset\mathcal{F}_n$ for $n\in\mathbb{N}_1$. It follows from Claim \ref{level_up} that $[b, b']\subset \hat{f}_\textnormal{NME}([p, q])$, where $[b, b']\subset\mathcal{F}_{n-1}$. Since $\diam([b, b'])\geq 4\diam([p, q])$, the slope $s_{\hat{f}_\textnormal{NME}}([p, q])\geq 4$. The proof is completed.
    \end{proof}
   
    \begin{claim}\label{f_hat_mixing}
        The map $\hat{f}_\textnormal{NME}$ is mixing.
    \end{claim}
    \begin{proof}
        Fix a nonempty open subset $I\subset\G$. Without loss of generality, assume that $I$ is an open subset of a $P_{\hat{f}_\textnormal{NME}}$-basic interval. Since $\hat{f}_\textnormal{NME}$ is linear on each $P$-basic interva and  $s_{\hat{f}_\textnormal{NME}}>2$ by Claim \ref{slope_hat_f_2}, there exists $n_0\in\mathbb{N}$ such that $(\hat{f}_\textnormal{NME})^{n_0}(I)$ contains at least $2$ points of $P_{\hat{f}_\textnormal{NME}}$, in particular $[p, q]\subset (\hat{f}_\textnormal{NME})^{n_0}(I)$ for some $P_{\hat{f}_\textnormal{NME}}$-basic interval $[p, q]$.

        We claim that there exists $n_1\in\mathbb{N}$ such that $(\hat{f}_\textnormal{NME})^{n_1}([p, q])$ contains a $P_g$-basic interval.  Assume that $[p, q]$ is not a $P_g$-basic interval. Image of any  $P_{\hat{f}_\textnormal{NME}}$-basic interval  contained in $T_M$, including intervals in $[p, q]\subset[p_n, q_n]\in\mathcal{J}_\textnormal{exit}$, covers $P_g$ interval in next iteration, so let us assume that $[p, q]\subset\G\setminus T_M$.

        If $[p, q]\subset\mathcal{F}_0$, then $\hat{f}_{\textnormal{NME}}([p, q])\subset T_M$ and if $[p, q]\subset\mathcal{F}_n$ for $n\geq 1$, then $(\hat{f}_\textnormal{NME})^{n-1}([p, q])\cap T_M\neq\emptyset$. 
        
        Finally, assume that $[p, q]\subset\G\setminus\mathcal{S}$. It follows that either $[p, q]\subset [b_{\omega}(k), b_{\omega'}(k)]$ or $[p, q]\subset I_k$ for some $\omega\in\{0, 1\}^n, \omega'\in\{0, 1\}^{n+1}, n\in\mathbb{N}$ and $k\in\{0,\dots, N'-1\}$. In the latter case, observe that $(\hat{f}_\textnormal{NME})^2([p, q])\cap T_M\neq\emptyset$, as we have that $I_{k+1\;\textrm{mod }  N'}\subset \hat{f}_{\textnormal{NME}}([p, q])$ and $\hat{f}_\textnormal{NME}(I_k)\cap T_M\neq \emptyset$ for any $k=0,\dots,N'-1$. If $[p, q]\subset [b_{\omega}(k), b_{\omega'}(k)]$, then $[b_{\tilde{\omega}}(m), b_{\tilde{\omega}'}(m)]\subset(\hat{f}_\textnormal{NME})^{N'2^n}([p, q])$, where $\tilde{\omega}\in\{0, 1\}^{n-1}, \tilde{\omega}'\in\{0, 1\}^n$ and $m\in\{0,\dots, N'-1\}$. By induction, there exists $l\in\mathbb{N}$ such that $[b_{\lambda}(N'-1), b_{1}(N'-1)]\subset(\hat{f}_\textnormal{NME})^l([p, q])$ and hence, $I_0\subset(\hat{f}_{\textnormal{NME}})^{l+1}([p, q])$, and this case has already been considered.
        Indeed, there exists $n_1\in\mathbb{N}$ such that $(\hat{f}_\textnormal{NME})^{n_1}([p, q])$ contains a $P_g$-basic interval.

        Since $g$ is exact on $T_M$  by Claim \ref{g_exact_markov} and we have that $g(J)\subset \hat{f}_\textnormal{NME}(J)$ for any $P_g$-basic interval $J$, then there exists $n_2\in\mathbb{N}$ such that $T_M\subset (\hat{f}_\textnormal{NME})^{n_0+n_1+n_2}([p, q])$. It follows that that $$T_M\cup\mathcal{F}_1\cup\bigcup_{k=0}^{N'-1}I_k\subset\hat{f}_\textnormal{NME}(T_M).$$

        Moreover, we have that $\mathcal{F}_{n-1}\cup\mathcal{F}_n\cup\mathcal{F}_{n+1}\subset\hat{f}_\textnormal{NME}(\mathcal{F}_n)$ for all $n\in\mathbb{N}_1$ (see Claim \ref{pietra_surjektywne}). Let 
        \begin{equation}\label{eq:B(n)}
        \mathcal{B}(n)=\{[b_{\omega}(k), b_{\omega'}(k)]\colon\;  \omega\in\{0, 1\}^{n-1}, \;\omega'\in\{0, 1\}^{n}, \; k=0,\dots, N'-1\}
        \end{equation}
        	be the set of all edges in $\G\setminus\mathcal{S}$ that are defined by a binary sequence $\omega$ of length $n\in\mathbb{N}_1$. One can see that $\mathcal{B}(n)\subset\hat{f}_\textnormal{NME}(\mathcal{B}(n))$. Furthermore, $\bigcup_{k=0}^{N'-1}I_k\subset \hat{f}_\textnormal{NME}(\bigcup_{k=0}^{N'-1}I_k)$ and $\bigcup_{k=0}^{N'-1}I_k\cup \mathcal{B}(1)\subset (\hat{f}_\textnormal{NME})^{2}(\bigcup_{k=0}^{N'-1}I_k)$, which follows straight from the definition of $\hat{f}_\textnormal{NME}$ on $I_k$.  Moreover, as the action of $\hat{f}_\textnormal{NME}$ follows the periodic orbit of length $N'$ in $\End(T_M)$ and the periodic action given by $\oplus$-addition operator on $\tilde{\G}$, we have that $\mathcal{B}(n+1)\subset(\hat{f}_\textnormal{NME})^{N'2^{n+1}}(\mathcal{B}(n))$. Thus, we have that $\mathcal{B}(n)\subset(\hat{f}_\textnormal{NME})^{k(n)}(\mathcal{B}(1))$, where $k(n)=\sum_{l=2}^n N'2^{l}$ for  $n\in\mathbb{N}_2$. Therefore, one can see that
        \begin{equation*}
            T_M\cup\bigcup_{m=M+1}^{M+n+2} \mathcal{L}_m\subset(\hat{f}_\textnormal{NME})^{k(n+1)+3}(T_M),
        \end{equation*}
        where $n\in\mathbb{N}_1$. Observe that for all $n\in\mathbb{N}_1$ we have $T_{M+n}\subset\hat{f}_{\textnormal{NME}}(T_{M+n})$ and also $$T_{M+n+2}\subset(\hat{f}_{\textnormal{NME}})^{n_0+n_1+n_2+k(n+1)+3}(I).$$ Hence, the map $\hat{f}_\textnormal{NME}$ is mixing.
    \end{proof}
     \begin{claim}\label{distance_NME}
        The distance $d_{\sup}(f, \hat{f}_\textnormal{NME})<\gamma$.
    \end{claim}
    \begin{proof}
    Observe first that $\hat{f}_\textnormal{NME}(x)=g(x)$ for all $x\in\overline{T_M\setminus\bigcup_{n=1}^{2^M}[p_n, q_n]}$ and hence, all estimates for $g$ apply also to $\hat{f}_\textnormal{NME}$. Thus,  by Claim \ref{diam_g_pq} we have that the distance $d(f(x), \hat{f}_\textnormal{NME}(x))< 9\epsilon$ for $x\in\overline{T_M\setminus\bigcup_{n=1}^{2^M}[p_n, q_n]}$.
    
    Let now $y\in [p_n, q_n]\in\mathcal{J}_\textnormal{exit}$. One can see that $g([p_n, q_n])\subset \hat{f}_\textnormal{NME}([p_n, q_n])$ and moreover $\hat{f}_\textnormal{NME}([p_n, q_n])\setminus g([p_n, q_n])\subset \G\setminus T_M$. By the choice of $M$ we have $\diam(C)<\delta$ for every connected component $C\subset \overline{\G\setminus T_M}$. We have that
    $$
        \begin{aligned}
            d(f(y), \hat{f}_\textnormal{NME}(y))&\leq d(f(y), g(y))+d(g(y), \hat{f}_\textnormal{NME}(y))\\
            &\leq 9\epsilon+\diam(C)+\diam(X_{[p_n, q_n]}),
        \end{aligned}
    $$
    which yet again follows from Claim \ref{diam_g_pq}. Since the diameter $\diam(X_{[p_n, q_n]})\leq \epsilon+4\delta$, we have that $d(f(y), \hat{f}_\textnormal{NME}(y))<15\epsilon$.

    Let now $C\subset \overline{\G\setminus T_M}$ be a connected component and $z\in C$. By the choice of $M, M_0$ there exists a connected component $Y\subset\G\setminus T_{M_0}$ such that $f(C)\subset Y$. Moreover, we have that $\hat{f}_\textnormal{NME}(C)\subset Y$ for the same connected component $Y$. Thus, we have that $d(f(z), \hat{f}_\textnormal{NME}(z))< 2\epsilon$.
    Combining all these arguments altogether,  we see that $d_{\sup}(f, \hat{f}_\textnormal{NME})<\gamma$.
    
    \end{proof}
    
    \begin{claim}\label{hat_f_NME_not_positive_recurrent}
        The Markovian graph $G_{\hat{f}_\textnormal{NME}}$ is not positive recurrent.
    \end{claim}
    \begin{proof}
        Firstly, let us estimate the entropy $h(G_{\hat{f}_\textnormal{NME}})$. One can see that for any $v\in V(G_{\hat{f}_\textnormal{NME}})$ the number of vertices $w\in V(G_{\hat{f}_\textnormal{NME}})$ such that $(v, w)\in E(G_{\hat{f}_\textnormal{NME}})$ is finite and bounded by $S'$, which follows straight from the choice of the parameter $S'$. Thus, by Theorem \ref{different_uv_same_limit_entropy} we have that
        \begin{equation*}
            h(G_{\hat{f}_\textnormal{NME}})=\lim_{n\to\infty}\frac{1}{n}\log(p_{vw}^{G_{\hat{f}_\textnormal{NME}}}(n))\leq \lim_{n\to\infty}\frac{\log (S')^n}{n}=\log S'.
        \end{equation*}
         Let now $H_n$ be a subgraph of $G_{\hat{f}_\textnormal{NME}}$ defined by taking all vertices representing $P_{\hat{f}_\textnormal{NME}}$-basic intervals inside element of $\mathcal{B}(n)$ for $n\in\mathbb{N}$ (see \eqref{eq:B(n)} in Claim \ref{f_hat_mixing} for the definition of $\mathcal{B}(n)$). By the construction, in each edge $[b, b']\subset \mathcal{B}(n)$ there are either $3$ or $S'$ $P_{\hat{f}_\textnormal{NME}}$-basic intervals. 
         
         let $H=\bigcup_{n=0}^\infty H_n$ and for each $n$, let
          $v_n$ be a vertex representing  a $P_{\hat{f}_\textnormal{NME}}$-basic interval inside the edge $[b_{1^{(n)}}(N'-1), b_{1^{(n+1)}}(N'-1)]\subset \mathcal{B}(n)$.  One can see that $H$ is strongly connected, which follows from the definition of $\hat{f}_\textnormal{NME}$. Furthermore, by Theorem \ref{different_uv_same_limit_entropy} we have that
        \begin{eqnarray*}
            h(H)&\geq& \lim_{n\to\infty}\frac{1}{N'2^{n+1}}\log(p_{v_{n}v_{n}}^{H_{n}}(N'2^{n+1}))\\
            &=&\lim_{n\to\infty}\frac{N'2\log 3}{N'2^{n+1}}+\frac{N'2^{n+1}-2}{N'2^{n+1}}  \log S'=\log S',
        \end{eqnarray*} as we consider only the loops of $v_n$ of length $N'2^{n+1}$ in each $H_n$ and we have that there are exactly $3^{N'2}(S')^{N'2^{n+1}-2}$ such loops that are distinct.  Thus, the graph $G_{\hat{f}_\textnormal{NME}}$ is not strongly positive recurrent by Theorem \ref{Ruette_subgraphs_SPR}. Moreover,
        $h(G_{\hat{f}_\textnormal{NME}})=\log S'$.
        
        By Theorem~\ref{different_uv_same_limit_entropy} we know that $R(G_{\hat{f}_\textnormal{NME}})^{-1}= S'$. Take the constant vector $\mathbf 1=(1)_{v\in V(G_{\hat{f}_\textnormal{NME}})}$, which is strictly positive and since the out-degree of every vertex is bounded by $S'$ we obtain that for every $v$,
        \begin{equation}\label{eq:sub}
\sum_{w}a_{vw}\mathbf 1_{w}\le S'=R(G_{\hat{f}_\textnormal{NME}})^{-1}\mathbf 1_{v}
        \end{equation}
        so the vector  $\mathbf 1$ is $R(G_{\hat{f}_\textnormal{NME}})$-subinvariant.
        But there are $v\in V(G_{\hat{f}_\textnormal{NME}})$ with the number of outgoing edges $\sum_{w\in V(G_{\hat{f}_\textnormal{NME}})}a_{vw}<S'$ so $\mathbf 1$ is not $R(G_{\hat{f}_\textnormal{NME}})$-invariant.
        Then by Theorem~\ref{Seneta:Rec} $G_{\hat{f}_\textnormal{NME}}$ is not positive recurrent (in fact $G_{\hat{f}_\textnormal{NME}}$ is transient).
        This completes the proof.
    \end{proof}

    \begin{claim}\label{entropy_endpoints_smaller}
        The system $(\G, \hat{f}_\textnormal{NME})$ does not admit a measure of maximal entropy.
    \end{claim}
    \begin{proof}
         We first examine the entropies of invariant subsystems in $(\G, \hat{f}_\textnormal{NME})$. We do so in the view of Theorem \ref{lipschitz_entropy_graf}.
        
         Observe that if a $P_{\hat{f}_\textnormal{NME}}$-basic interval $[p, q]\cap \bigcup_{n=1}^\infty \mathcal{B}(n)=\emptyset$, then the slope $s_{\hat{f}_\textnormal{NME}}([p, q])\leq L$ (see Claims \ref{L-lipschitz_constants}, \ref{f_hat_mixing} for details on $L$ and $\mathcal{B}(n)$). Moreover, if a $P_{\hat{f}_\textnormal{NME}}$-basic interval $[p, q]$ is contained in either $[b_{1^{(n)}}(N'-1), b_{1^{(n+1)}}(N'-1)]$, $[b_{0^{(n)}}(0), b_{0^{(n+1)}}(0)]$ for some $n\in\mathbb{N}$, then the slope $s_{\hat{f}_\textnormal{NME}}([p, q])<36$, as we have that
        \begin{equation*}
            \diam(\hat{f}_\textnormal{NME}([p, q]))< 3\cdot 4\diam([b_{0^{(n)}}(0), b_{0^{(n+1)}}(0)])
        \end{equation*}
        and also
        \begin{equation*}
            \begin{split}
                \diam([b_{1^{(n)}}(N'-1), b_{1^{(n+1)}}(N'-1)])&=\diam([b_{0^{(n)}}(0), b_{0^{(n+1)}}(0)])\\
                &=3\diam([p, q]).
            \end{split}
        \end{equation*}
        Recall that $L>36$ by the choice of parameter $d=M-M_0$ and the parameter $S'$ has been chosen so that $L<S'$. We used this $S'$ to set $S'$-folds on $$\bigcup_{n=0}^\infty  \mathcal{B}(n+1)\setminus [b_{1^{(n)}}(N'-1), b_{1^{(n+1)}}(N'-1)]\cup [b_{0^{(n)}}(0), b_{0^{(n+1)}}(0)].$$
        
        Observe that if we define the map $\hat{f}'_\textnormal{NME}$ to be exactly the same as $\hat{f}_\textnormal{NME}$ on $\mathcal{S}$, but on $\G\setminus\mathcal{S}$ instead of setting $S'$-folds we would set $1$-folds (or any $K$-folds with an odd $K<L$), then $\hat{f}'_\textnormal{NME}=\hat{f}_\textnormal{NME}$ on $B(\G)$, as we would not alter the action on the branching points. Consequently, $\hat{f}'_\textnormal{NME}=\hat{f}_\textnormal{NME}$ on $\End(\G)$, which is an invariant subset for both maps. Moreover, $\hat{f}'_\textnormal{NME}$ is $L$-Lipschitz, which follows from a similar argument to the one presented in the proof of Claim \ref{L-lipschitz_argument}. Hence, we have that
        \begin{equation*}
            \htop(\End(\G), \hat{f}_\textnormal{NME}|_{\End(\G)})=\htop(\End(\G), \hat{f}'_\textnormal{NME}|_{\End(\G)})\leq\log L.
        \end{equation*}
        On the other hand, since $h(G_{\hat{f}_\textnormal{NME}})=\log S'$ and $\hat{f}_\textnormal{NME}$ is a mixing Markovian map, we have that $\htop(\G, \hat{f}_\textnormal{NME})=\log S'$ by Corollary \ref{entropia_odwrocone_corollary}.
        
         Therefore, the system $(\G, \hat{f}_\textnormal{NME})$ must not admit any measure of maximal entropy, which follows from Corollary \ref{MME_FINAL_CONDITION}
         combined with Claim~\ref{hat_f_NME_not_positive_recurrent}.
    \end{proof}

\begin{step}
    Modifying the map $\hat{f}_\textnormal{NME}$ to $\hat{f}_\textnormal{MME}$.
\end{step}

Since $(\G, \hat{f}_\textnormal{NME})$ is mixing, the graph $G_{\hat{f}_\textnormal{NME}}$ is strongly connected. Thus, there exists a finite loop $\overline{L}=(v_n)_{n=0}^{K-1}\subset V(G_{\hat{f}_\textnormal{NME}})$ for some $K\in\mathbb{N}_1$ with $\{(v_n, v_{n+1\;\textrm{mod }K})\colon \; n=0,\dots, K-1\}\subset E(G_{\hat{f}_\textnormal{NME}})$. Let $S''$ be an odd integer such that $S'<S''$. Let $I_n$ be the $P_{\hat{f}_\textnormal{NME}}$-basic interval identified with vertex $v_n\in C$ for each $n$. 

Define $\hat{f}_\textnormal{MME}(x)=\hat{f}_\textnormal{NME}(x)$ for all $x\in\overline{\G\setminus\bigcup_{n=0}^{K-1}I_n}$. Then in $I_n$ define  $\hat{f}_\textnormal{MME}$ as an $S''$-fold between $I_n, \hat{f}_\textnormal{NME}(I_n)$, where $n=0,\dots, K-1$. We have to extend the partition by including newly obtained linearity intervals. If $\hat{f}_\textnormal{NME}(I_n)=[x, y]$, then denote by $P(I_n)=\{z\in I_n\colon\; \hat{f}_\textnormal{MME}(z)\in\{x, y\}\}$. One can see that there are exactly $S''+1$ such points and those are the endpoints of subintervals in $I_n$ on which the map $\hat{f}_\textnormal{MME}$ is linear. Let 
$$
P_{\hat{f}_\textnormal{MME}}=P_{\hat{f}_\textnormal{NME}}\cup\bigcup_{n=0}^{K-1}P(I_n).
$$
One can see that the map $\hat{f}_\textnormal{MME}$ is mixing and $P_{\hat{f}_\textnormal{MME}}$-Markovian, which follows from similar arguments to those presented in Claims \ref{f_hat_markovian}, \ref{f_hat_mixing}. Moreover, the entropy $h(G_{\hat{f}_\textnormal{MME}})\leq \log (S'+S'')$. Indeed, the graph $G_{\hat{f}_\textnormal{MME}}$ is strongly connected and we use Theorem \ref{different_uv_same_limit_entropy} in the same way as in Claim \ref{hat_f_NME_not_positive_recurrent}, where we note that now there are at most $S'+ S''$ outgoing edges from any vertex $v\in V(G_{\hat{f}_\textnormal{MME}})$.

\begin{claim}
    The distance $d_{\sup}(f, \hat{f}_\textnormal{MME})<\gamma$.
\end{claim}
\begin{proof}
    The following claim is a consequence of Claim \ref{distance_NME} and hence, our reasoning here will be similar to its proof. Observe that if $x\in \overline{\G\setminus\bigcup_{n=0}^{K-1} I_n}$, then $$d(f(x), \hat{f}_\textnormal{MME}(x))=d(f(x), \hat{f}_\textnormal{NME}(x))<\gamma.$$ Let $x\in \bigcup_{n=0}^{K-1} I_n$. We have that
    \begin{equation*}
        d(f(x), \hat{f}_\textnormal{MME}(x))\leq d(f(x), \hat{f}_\textnormal{NME}(x))+d(\hat{f}_\textnormal{NME}(x), \hat{f}_\textnormal{MME}(x))
    \end{equation*}
    and hence, we need to estimate the distance $d(\hat{f}_\textnormal{NME}(x), \hat{f}_\textnormal{MME}(x))$, since we have that $d(f(x), \hat{f}_\textnormal{NME}(x))<15\epsilon$. Assume first that $x\in[p, q]$, where the interval $[p, q]\subset \overline{T_M\setminus\bigcup_{n=1}^{2^M}[p_n, q_n]}$ is a $P_g$-basic interval. Then 
    \begin{equation*}
        d(\hat{f}_\textnormal{NME}(x), \hat{f}_\textnormal{MME}(x))\leq \diam(X_{[p, q]})\leq \epsilon+4\delta\leq 5\epsilon.
    \end{equation*}
    Assume now that $x\in[p_n, q_n]$ for some $n\in\{0,\dots, 2^M\}$, then 
    \begin{equation*}
        d(\hat{f}_\textnormal{NME}(x), \hat{f}_\textnormal{MME}(x))\leq \diam(X_{[p, q]})+\diam( C)\leq \epsilon+5\delta\leq 6\epsilon,
    \end{equation*}
    where $ C\subset \overline{\G\setminus T_M}$ is a connected component. Finally, assume that $x\in\overline{\G\setminus T_{M}}$. Then we have
    \begin{equation*}
        d(\hat{f}_\textnormal{NME}(x), \hat{f}_\textnormal{MME}(x))\leq \diam( D)\leq 2\epsilon,
    \end{equation*}
    where $ D\subset \overline{\G\setminus T_{M_0}}$ is a connected component. Thus $d(f(x), \hat{f}_\textnormal{MME}(x))\leq 21\epsilon$. Hence, $d(f(x), \hat{f}_\textnormal{MME}(x))<\gamma$.
\end{proof}

Recall that $t^{W}_{uv}(n)$ for a subgraph $W\subset G$ was defined in Definition~\ref{def:tuv}.
\begin{claim}
    Let $W( \overline{L})$ be a graph composed of vertices associated to $P_{\hat{f}_\textnormal{MME}}$-basic intervals in $\bigcup_{n=0}^{K-1} I_n$. We have that
    \begin{equation*}
        \limsup_{n\to\infty}\frac{1}{n}\log t^{W( \overline{L})^c}_{uv}(n)\leq h(W( \overline{L})).
    \end{equation*}
\end{claim}
\begin{proof}
    Fix $u, v\in W( \overline{L})$. One can see that there are at most $S'$ edges that go from $u$ to vertices in $W(\overline{L})^c$ and there is only one edge that goes from  any of the vertices in $W( \overline{L})^c$ to v. Similarly, there are at most $S'$ outgoing edges from any vertex in $W( \overline{L})^c$. Thus, we have that
    \begin{equation*}
        \limsup_{n\to\infty}\frac{1}{n}\log t^{W(\overline{L})^c}_{uv}(n)\leq \limsup_{n\to\infty}\frac{1}{n}\log (S')^{n-1}=\log S'.
    \end{equation*}
    On the other hand, for any vertex in $W( \overline{L})$ there are exactly $S''$ outgoing edges, as we introduced $S''$-folds on the $P_{\hat{f}_\textnormal{NME}}$-basic intervals associated to the vertices forming the loop $ \overline{L}$. Since $W( \overline{L})$ is connected, by Theorem \ref{different_uv_same_limit_entropy} for any vertex $w\in W( \overline{L})$ we have
    \begin{equation*}
        h(W( \overline{L}))=\lim_{n\to\infty}\frac{1}{nK}\log(p_{ww}^{W( \overline{L})}(nK))=\lim_{n\to\infty}\frac{1}{nK}\log(S'')^{nK-1}=\log S''.
    \end{equation*}
    Thus, the assertion follows.
\end{proof}
Thus, the graph $G_{\hat{f}_\textnormal{MME}}$ is positive recurrent by Theorem \ref{SPR_high_local_entropy} and $$\log S''\leq h(G_{\hat{f}_\textnormal{MME}})\leq \log (S'+S'').$$ 
Additionally, the topological entropy $\htop(\End(\G), \hat{f}_\textnormal{MME}|_{\End(\G)})< \htop(\G, \hat{f}_\textnormal{MME})$, which follows from a similar argument as in Claim \ref{entropy_endpoints_smaller}. Therefore, the system $(\G, \hat{f}_\textnormal{MME})$ admits a unique measure of maximal entropy by Corollary \ref{MME_FINAL_CONDITION}.
\end{proof}

\section{Discussion of the results}
The constructions above are tailored to the Gehman dendrite, but several ingredients are not specific to the binary branching model. The argument mainly uses the existence of arbitrarily fine self-similar structure in reconstructing the map over the dendrite $\mathcal{S}$, a closed endpoint set for it to become a negligible boundary subsystem, and a uniform bound on the order of branch points used in Claims \ref{L-lipschitz_constants} and \ref{claim_8}. For this reason, we expect analogous approximation phenomena for a broader class of dendrites. A complete formulation would require additional arguments, and we leave this outside the present paper.

\begin{question}
    Let $(X, f)$ be a transitive topological dynamical system on a dendrite $X$ whose set of branch points is discrete and of bounded order.
    Can $(X,f)$ be approximated by a map with an MME and by a map without an MME?
\end{question}

One limitation of the present construction is that it does not attempt to preserve entropy. In order to obtain mixing, which is the key input for the MME assertions in Section \ref{MMEs_assertions}, we force local omnidirectional expansion. This increases the topological entropy of the approximating model and leads to the following natural question.

\begin{question}
    Let $(X, f)$ be a transitive topological dynamical system on the Gehman dendrite $X$. Does there exist a dynamical system $(X, g_\epsilon)$ that admits a measure of maximal entropy and both $d_{\sup}(f, g_\epsilon)<\epsilon$ and $|\htop(X, f)-\htop(X, g_\epsilon)|<\epsilon$ for any initially chosen parameter $\epsilon>0$?
\end{question}

Finally, we conclude the paper by asking another interesting question.

\begin{question}
    Does there exist a characterization of Markovian dynamical systems that necessarily admit an MME? In other words, does there exist a dynamical property of a Markovian system that forces its Markovian graph to be positive recurrent?
\end{question}

\section{Acknowledgements}
The authors are grateful to Dominik Kwietniak, as many of the ideas, notations, and techniques in this paper developed from previous joint work.

This work was partially supported by project
No.~CZ.02.01.01/00/23\_021/0008759 supported by EU funds, through the Operational Programme Johannes Amos Comenius.

\bibliographystyle{plain}
\bibliography{references}

@book {Seneta,
	AUTHOR = {Seneta, E.},
	TITLE = {Non-negative matrices and {M}arkov chains},
	SERIES = {Springer Series in Statistics},
	NOTE = {Revised reprint of the second (1981) edition [Springer-Verlag,
	New York; MR0719544]},
	PUBLISHER = {Springer, New York},
	YEAR = {2006},
	PAGES = {xvi+287},
	ISBN = {978-0387-29765-1; 0-387-29765-0},
	MRCLASS = {60-02 (15A48 60J10)},
	MRNUMBER = {2209438},
}

@article {Ruette,
    AUTHOR = {Ruette, Sylvie},
     TITLE = {On the {V}ere-{J}ones classification and existence of maximal
              measures for countable topological {M}arkov chains},
   JOURNAL = {Pacific J. Math.},
  FJOURNAL = {Pacific Journal of Mathematics},
    VOLUME = {209},
      YEAR = {2003},
    NUMBER = {2},
     PAGES = {366--380},
      ISSN = {0030-8730,1945-5844},
   MRCLASS = {37B10 (37A35 37A50 37B20 60J10)},
  MRNUMBER = {1978377},
MRREVIEWER = {J\'er\^ome\ Buzzi},
       DOI = {10.2140/pjm.2003.209.365},
}

@article {Vere,
    AUTHOR = {Vere-Jones, D.},
     TITLE = {Geometric ergodicity in denumerable {M}arkov chains},
   JOURNAL = {Quart. J. Math. Oxford Ser. (2)},
  FJOURNAL = {The Quarterly Journal of Mathematics. Oxford. Second Series},
    VOLUME = {13},
      YEAR = {1962},
     PAGES = {7--28},
      ISSN = {0033-5606,1464-3847},
   MRCLASS = {60.65},
  MRNUMBER = {141160},
MRREVIEWER = {J.\ S.\ Griffin, Jr.},
       DOI = {10.1093/qmath/13.1.7},
}

@article {Gurevich_infinite_graphs_no_MME,
    AUTHOR = {Gurevich, B. M.},
     TITLE = {Topological entropy of a countable {M}arkov chain},
   JOURNAL = {Dokl. Akad. Nauk SSSR},
  FJOURNAL = {Doklady Akademii Nauk SSSR},
    VOLUME = {187},
      YEAR = {1969},
     PAGES = {715--718},
      ISSN = {0002-3264},
   MRCLASS = {60.65 (28.00)},
  MRNUMBER = {263162},
MRREVIEWER = {I.\ Csisz\'{a}r},
}

@article {Gurevich_MME_iff,
    AUTHOR = {Gurevich, B. M.},
     TITLE = {Shift entropy and {M}arkov measures in the space of paths of a
              countable graph},
   JOURNAL = {Dokl. Akad. Nauk SSSR},
  FJOURNAL = {Doklady Akademii Nauk SSSR},
    VOLUME = {192},
      YEAR = {1970},
     PAGES = {963--965},
      ISSN = {0002-3264},
   MRCLASS = {28.70 (54.00)},
  MRNUMBER = {268356},
MRREVIEWER = {E.\ M.\ Klimko},
}

@article {Gurevich_Zargaryan,
    AUTHOR = {Gurevich, B. M. and Zargaryan, A. S.},
     TITLE = {Conditions for the existence of a maximal measure for a
              countable symbolic {M}arkov chain},
   JOURNAL = {Vestnik Moskov. Univ. Ser. I Mat. Mekh.},
  FJOURNAL = {Vestnik Moskovskogo Universiteta. Seriya I. Matematika,
              Mekhanika},
      YEAR = {1988},
    NUMBER = {5},
     PAGES = {14--18, 103},
      ISSN = {0579-9368},
   MRCLASS = {58F11 (28D05)},
  MRNUMBER = {1051173},
}

@article {Charatonik,
    AUTHOR = {Ar\'evalo, Daniel and Charatonik, W\l odzimierz J. and
              Pellicer Covarrubias, Patricia and Sim\'on, Likin},
     TITLE = {Dendrites with a closed set of end points},
   JOURNAL = {Topology Appl.},
  FJOURNAL = {Topology and its Applications},
    VOLUME = {115},
      YEAR = {2001},
    NUMBER = {1},
     PAGES = {1--17},
      ISSN = {0166-8641,1879-3207},
   MRCLASS = {54F15 (54C10 54F50)},
  MRNUMBER = {1840729},
MRREVIEWER = {Eldon\ J.\ Vought},
       DOI = {10.1016/S0166-8641(00)00058-4},
}

@incollection {Salama_entropy_R,
    AUTHOR = {Salama, Ibrahim A.},
     TITLE = {On the recurrence of countable topological {M}arkov chains},
 BOOKTITLE = {Symbolic dynamics and its applications ({N}ew {H}aven, {CT},
              1991)},
    SERIES = {Contemp. Math.},
    VOLUME = {135},
     PAGES = {349--360},
 PUBLISHER = {Amer. Math. Soc., Providence, RI},
      YEAR = {1992},
      ISBN = {0-8218-5146-2},
   MRCLASS = {54H20 (28D05)},
  MRNUMBER = {1185102},
MRREVIEWER = {Manfred\ Denker},
       DOI = {10.1090/conm/135/1185102},
}

@article {Entropy_paradox_drzewa,
    AUTHOR = {Hara\'nczyk, Grzegorz and Kwietniak, Dominik and Oprocha,
              Piotr},
     TITLE = {Topological structure and entropy of mixing graph maps},
   JOURNAL = {Ergodic Theory Dynam. Systems},
  FJOURNAL = {Ergodic Theory and Dynamical Systems},
    VOLUME = {34},
      YEAR = {2014},
    NUMBER = {5},
     PAGES = {1587--1614},
      ISSN = {0143-3857,1469-4417},
   MRCLASS = {37F25 (37B20)},
  MRNUMBER = {3255434},
       DOI = {10.1017/etds.2013.6},
}

@misc{Entropy_paradox_dendryty,
  doi = {10.48550/ARXIV.2504.05121},
  author = {Kwietniak,  Dominik and Oprocha,  Piotr and Tomaszewski,  Jakub},
  title = {On entropy of pure mixing maps on dendrites},
  publisher = {arXiv},
  year = {2025},
  copyright = {arXiv.org perpetual,  non-exclusive license}
}

@article {Parry,
    AUTHOR = {Parry, William},
     TITLE = {Intrinsic {M}arkov chains},
   JOURNAL = {Trans. Amer. Math. Soc.},
  FJOURNAL = {Transactions of the American Mathematical Society},
    VOLUME = {112},
      YEAR = {1964},
     PAGES = {55--66},
      ISSN = {0002-9947,1088-6850},
   MRCLASS = {60.65},
  MRNUMBER = {161372},
MRREVIEWER = {H.\ P.\ Edmundson},
       DOI = {10.2307/1994009},
       URL = {https://doi.org/10.2307/1994009},
}

@article {Misiurewicz_hausdorff,
    AUTHOR = {Misiurewicz, Micha\l{}},
     TITLE = {On Bowen's definition of topological entropy},
   JOURNAL = {Discrete Contin. Dyn. Syst.},
  FJOURNAL = {Discrete and Continuous Dynamical Systems. Series A},
    VOLUME = {10},
      YEAR = {2004},
    NUMBER = {3},
     PAGES = {827--833},
      ISSN = {1078-0947,1553-5231},
   MRCLASS = {37B40 (37C45 54C70 54F45)},
  MRNUMBER = {2018882},
MRREVIEWER = {Michael\ Hurley},
       DOI = {10.3934/dcds.2004.10.827},
       URL = {https://doi.org/10.3934/dcds.2004.10.827},
}

@article {Alseda,
    AUTHOR = {Alsed\`a, Ll. and del R\'io, M. A. and Rodr\'iguez, J. A.},
     TITLE = {A splitting theorem for transitive maps},
   JOURNAL = {J. Math. Anal. Appl.},
  FJOURNAL = {Journal of Mathematical Analysis and Applications},
    VOLUME = {232},
      YEAR = {1999},
    NUMBER = {2},
     PAGES = {359--375},
      ISSN = {0022-247X,1096-0813},
   MRCLASS = {54H20 (54C70)},
  MRNUMBER = {1683124},
MRREVIEWER = {Romeo\ F.\ Thomas},
       DOI = {10.1006/jmaa.1999.6277},
       URL = {https://doi.org/10.1006/jmaa.1999.6277},
}

@article {Positive_entropy_Gehman,
    AUTHOR = {Dirb\'{a}k, Mat\'{u}\v{s} and Snoha, L'ubom\'{\i}r and
              \v{S}pitalsk\'{y}, Vladim\'{\i}r},
     TITLE = {Minimality, transitivity, mixing and topological entropy on
              spaces with a free interval},
   JOURNAL = {Ergodic Theory Dynam. Systems},
  FJOURNAL = {Ergodic Theory and Dynamical Systems},
    VOLUME = {33},
      YEAR = {2013},
    NUMBER = {6},
     PAGES = {1786--1812},
      ISSN = {0143-3857,1469-4417},
   MRCLASS = {37B40 (37A25)},
  MRNUMBER = {3122152},
MRREVIEWER = {Xueting\ Tian},
       DOI = {10.1017/S0143385712000442},
       URL = {https://doi.org/10.1017/S0143385712000442},
}

@book {Walters,
    AUTHOR = {Walters, Peter},
     TITLE = {An introduction to ergodic theory},
    SERIES = {Graduate Texts in Mathematics},
    VOLUME = {79},
 PUBLISHER = {Springer-Verlag, New York-Berlin},
      YEAR = {1982},
     PAGES = {ix+250},
      ISBN = {0-387-90599-5},
   MRCLASS = {28Dxx (54H20 58F11)},
  MRNUMBER = {648108},
MRREVIEWER = {M.\ A.\ Akcoglu},
}

@article {Ruette_natural_extension,
    AUTHOR = {Ruette, Sylvie},
     TITLE = {Mixing {$C^r$} maps of the interval without maximal measure},
   JOURNAL = {Israel J. Math.},
  FJOURNAL = {Israel Journal of Mathematics},
    VOLUME = {127},
      YEAR = {2002},
     PAGES = {253--277},
      ISSN = {0021-2172,1565-8511},
   MRCLASS = {37E05 (37A25 37A35 37B40)},
  MRNUMBER = {1900702},
MRREVIEWER = {Adriana\ Berechet},
       DOI = {10.1007/BF02784534},
       URL = {https://doi.org/10.1007/BF02784534},
}

@book {Petersen,
    AUTHOR = {Petersen, Karl},
     TITLE = {Ergodic theory},
    SERIES = {Cambridge Studies in Advanced Mathematics},
    VOLUME = {2},
      NOTE = {Corrected reprint of the 1983 original},
 PUBLISHER = {Cambridge University Press, Cambridge},
      YEAR = {1989},
     PAGES = {xii+329},
      ISBN = {0-521-38997-6},
   MRCLASS = {28Dxx (28-02 47A35 54H20 58F11)},
  MRNUMBER = {1073173},
MRREVIEWER = {Nathaniel\ Friedman},
}

@article {Baldwin,
    AUTHOR = {Baldwin, Stewart},
     TITLE = {Entropy estimates for transitive maps on trees},
   JOURNAL = {Topology},
  FJOURNAL = {Topology. An International Journal of Mathematics},
    VOLUME = {40},
      YEAR = {2001},
    NUMBER = {3},
     PAGES = {551--569},
      ISSN = {0040-9383},
   MRCLASS = {37B40 (37E25 54C70 54H20)},
  MRNUMBER = {1838995},
MRREVIEWER = {Joan\ Carles\ Tatjer},
       DOI = {10.1016/S0040-9383(99)00074-9},
       URL = {https://doi.org/10.1016/S0040-9383(99)00074-9},
}

@book {Parthasarathy,
    AUTHOR = {Parthasarathy, K. R.},
     TITLE = {Probability measures on metric spaces},
    SERIES = {Probability and Mathematical Statistics},
    VOLUME = {No. 3},
 PUBLISHER = {Academic Press, Inc., New York-London},
      YEAR = {1967},
     PAGES = {xi+276},
   MRCLASS = {60.00},
  MRNUMBER = {226684},
MRREVIEWER = {R.\ A.\ Gangolli},
}

@book {Kechris,
    AUTHOR = {Kechris, Alexander S.},
     TITLE = {Classical descriptive set theory},
    SERIES = {Graduate Texts in Mathematics},
    VOLUME = {156},
 PUBLISHER = {Springer-Verlag, New York},
      YEAR = {1995},
     PAGES = {xviii+402},
      ISBN = {0-387-94374-9},
   MRCLASS = {03E15 (03-01 03-02 04A15 28A05 54H05 90D44)},
  MRNUMBER = {1321597},
MRREVIEWER = {Jakub\ Jasi\'nski},
       DOI = {10.1007/978-1-4612-4190-4},
       URL = {https://doi.org/10.1007/978-1-4612-4190-4},
}

@book {Nadler,
    AUTHOR = {Nadler, Jr., Sam B.},
     TITLE = {Continuum theory},
    SERIES = {Monographs and Textbooks in Pure and Applied Mathematics},
    VOLUME = {158},
      NOTE = {An introduction},
 PUBLISHER = {Marcel Dekker, Inc., New York},
      YEAR = {1992},
     PAGES = {xiv+328},
      ISBN = {0-8247-8659-9},
   MRCLASS = {54-02 (54C05 54E45 54F15 54F50)},
  MRNUMBER = {1192552},
MRREVIEWER = {Hidefumi\ Katsuura},
}

@misc{Buzzi2025,
  doi = {10.48550/ARXIV.2501.07455},
  url = {https://arxiv.org/abs/2501.07455},
  author = {Buzzi,  Jér\^ome and Crovisier,  Sylvain and Sarig,  Omri},
  title = {Strong positive recurrence and exponential mixing for diffeomorphisms},
  publisher = {arXiv},
  year = {2025},
  copyright = {arXiv.org perpetual,  non-exclusive license}
}

@article{Newhouse1989,
  title = {Continuity Properties of Entropy},
  volume = {129},
  ISSN = {0003-486X},
  url = {http://dx.doi.org/10.2307/1971492},
  DOI = {10.2307/1971492},
  number = {1},
  journal = {The Annals of Mathematics},
  publisher = {JSTOR},
  author = {Newhouse,  Sheldon E.},
  year = {1989},
  month = Jan,
  pages = {215}
}

@article{Buzzi2022,
  title = {Measures of maximal entropy for surface diffeomorphisms},
  volume = {195},
  ISSN = {0003-486X},
  url = {http://dx.doi.org/10.4007/annals.2022.195.2.2},
  DOI = {10.4007/annals.2022.195.2.2},
  number = {2},
  journal = {Annals of Mathematics},
  publisher = {Annals of Mathematics},
  author = {Buzzi,  Jér\^ome and Crovisier,  Sylvain and Sarig,  Omri},
  year = {2022},
  month = Mar 
}

@article{Gouzel2010,
  title = {Almost sure invariance principle for dynamical systems by spectral methods},
  volume = {38},
  ISSN = {0091-1798},
  url = {http://dx.doi.org/10.1214/10-AOP525},
  DOI = {10.1214/10-aop525},
  number = {4},
  journal = {The Annals of Probability},
  publisher = {Institute of Mathematical Statistics},
  author = {Gouëzel,  Sébastien},
  year = {2010},
  month = {July} 
}

@article{Hardy1988,
     author = {Guivarc'h, Y. and Hardy, J.},
     title = {Th\'eor\`emes limites pour une classe de cha{\^\i}nes de {Markov} et applications aux diff\'eomorphismes {d'Anosov}},
     journal = {Annales de l'I.H.P. Probabilit\'es et statistiques},
     pages = {73--98},
     year = {1988},
     publisher = {Gauthier-Villars},
     volume = {24},
     number = {1},
     mrnumber = {937957},
     zbl = {0649.60041},
     language = {fr},
     url = {https://www.numdam.org/item/AIHPB_1988__24_1_73_0/}
}

@article{Kifer1990,
  title = {Large Deviations in Dynamical Systems and Stochastic Processes},
  volume = {321},
  ISSN = {0002-9947},
  url = {http://dx.doi.org/10.2307/2001571},
  DOI = {10.2307/2001571},
  number = {2},
  journal = {Transactions of the American Mathematical Society},
  publisher = {JSTOR},
  author = {Kifer,  Yuri},
  year = {1990},
  month = Oct,
  pages = {505}
}

@book{Pollicott1990,
     author = {Parry, William and Pollicott, Mark},
     title = {Zeta functions and the periodic orbit structure of hyperbolic dynamics},
     series = {Ast\'erisque},
     year = {1990},
     publisher = {Soci\'et\'e math\'ematique de France},
     number = {187-188},
     zbl = {0726.58003},
     language = {en},
     url = {https://www.numdam.org/item/AST_1990__187-188__1_0/}
}

@book{Ruelle2004,
  title = {Thermodynamic Formalism: The Mathematical Structure of Equilibrium Statistical Mechanics},
  ISBN = {9780511617546},
  url = {http://dx.doi.org/10.1017/CBO9780511617546},
  DOI = {10.1017/cbo9780511617546},
  publisher = {Cambridge University Press},
  author = {Ruelle,  David},
  year = {2004},
  month = Nov 
}

@article{Spitalsky2013,
     AUTHOR = {\v{S}pitalsk\'{y}, Vladim\'{\i}r},
     TITLE = {Topological entropy of transitive dendrite maps},
   JOURNAL = {Ergodic Theory Dynam. Systems},
  FJOURNAL = {Ergodic Theory and Dynamical Systems},
    VOLUME = {35},
      YEAR = {2015},
    NUMBER = {4},
     PAGES = {1289--1314},
      ISSN = {0143-3857,1469-4417},
   MRCLASS = {37B40 (37E25 54F50)},
  MRNUMBER = {3345173},
MRREVIEWER = {Marek\ Lampart},
       DOI = {10.1017/etds.2013.97},
       URL = {https://doi.org/10.1017/etds.2013.97},
}

@article{BARTO2019,
    AUTHOR = {Barto\v{s}, Adam and Bobok, Jozef and Pyrih, Pavel and Roth,
              Samuel and Vejnar, Benjamin},
     TITLE = {Constant slope, entropy, and horseshoes for a map on a tame
              graph},
   JOURNAL = {Ergodic Theory Dynam. Systems},
  FJOURNAL = {Ergodic Theory and Dynamical Systems},
    VOLUME = {40},
      YEAR = {2020},
    NUMBER = {11},
     PAGES = {2970--2994},
      ISSN = {0143-3857,1469-4417},
   MRCLASS = {37B45 (37B40 37E25)},
  MRNUMBER = {4157471},
MRREVIEWER = {Antonio\ Linero Bas},
       DOI = {10.1017/etds.2019.29},
       URL = {https://doi.org/10.1017/etds.2019.29}
}

@article{Bobok2016,
  title = {Constant Slope Maps and the Vere-Jones Classification},
  volume = {18},
  ISSN = {1099-4300},
  url = {http://dx.doi.org/10.3390/e18060234},
  DOI = {10.3390/e18060234},
  number = {6},
  journal = {Entropy},
  publisher = {MDPI AG},
  author = {Bobok,  Jozef and Bruin,  Henk},
  year = {2016},
  month = {June},
  pages = {234}
}

@article {Bobok2019,
    AUTHOR = {Bobok, Jozef and Roth, Samuel},
     TITLE = {The infimum of {L}ipschitz constants in the conjugacy class of
              an interval map},
   JOURNAL = {Proc. Amer. Math. Soc.},
  FJOURNAL = {Proceedings of the American Mathematical Society},
    VOLUME = {147},
      YEAR = {2019},
    NUMBER = {1},
     PAGES = {255--269},
      ISSN = {0002-9939,1088-6826},
   MRCLASS = {37E05 (26A16 37B40)},
  MRNUMBER = {3876747},
MRREVIEWER = {Steven\ M.\ Pederson},
       DOI = {10.1090/proc/14255},
       URL = {https://doi.org/10.1090/proc/14255},
}

@article {Buzzi2006,
    AUTHOR = {Buzzi, J\'{e}r\^{o}me and Ruette, Sylvie},
     TITLE = {Large entropy implies existence of a maximal entropy measure
              for interval maps},
   JOURNAL = {Discrete Contin. Dyn. Syst.},
  FJOURNAL = {Discrete and Continuous Dynamical Systems},
    VOLUME = {14},
      YEAR = {2006},
    NUMBER = {4},
     PAGES = {673--688},
      ISSN = {1078-0947,1553-5231},
   MRCLASS = {37E05 (37A35 37B40 37C40)},
  MRNUMBER = {2177091},
MRREVIEWER = {Jos\'{e}\ S.\ C\'{a}novas Pe\~{n}a},
       DOI = {10.3934/dcds.2006.14.673},
       URL = {https://doi.org/10.3934/dcds.2006.14.673},
}

@article {Boyle2006,
    AUTHOR = {Boyle, Mike and Buzzi, Jerome and G\'{o}mez, Ricardo},
     TITLE = {Almost isomorphism for countable state {M}arkov shifts},
   JOURNAL = {J. Reine Angew. Math.},
  FJOURNAL = {Journal f\"{u}r die Reine und Angewandte Mathematik. [Crelle's
              Journal]},
    VOLUME = {592},
      YEAR = {2006},
     PAGES = {23--47},
      ISSN = {0075-4102,1435-5345},
   MRCLASS = {37B10 (28D20 37A25 37B40 37D25)},
  MRNUMBER = {2222728},
MRREVIEWER = {Heber\ Enrich},
       DOI = {10.1515/CRELLE.2006.021},
       URL = {https://doi.org/10.1515/CRELLE.2006.021},
}

@article {Thermo1,
    AUTHOR = {Buzzi, J\'{e}r\^{o}me and Sarig, Omri},
     TITLE = {Uniqueness of equilibrium measures for countable {M}arkov
              shifts and multidimensional piecewise expanding maps},
   JOURNAL = {Ergodic Theory Dynam. Systems},
  FJOURNAL = {Ergodic Theory and Dynamical Systems},
    VOLUME = {23},
      YEAR = {2003},
    NUMBER = {5},
     PAGES = {1383--1400},
      ISSN = {0143-3857,1469-4417},
   MRCLASS = {37D35 (37A05 37B40 37F15)},
  MRNUMBER = {2018604},
MRREVIEWER = {Heber\ Enrich},
       DOI = {10.1017/S0143385703000087},
       URL = {https://doi.org/10.1017/S0143385703000087},
}

@article {Thermo2,
    AUTHOR = {Fiebig, Doris and Fiebig, Ulf-Rainer and Yuri, Michiko},
     TITLE = {Pressure and equilibrium states for countable state {M}arkov
              shifts},
   JOURNAL = {Israel J. Math.},
  FJOURNAL = {Israel Journal of Mathematics},
    VOLUME = {131},
      YEAR = {2002},
     PAGES = {221--257},
      ISSN = {0021-2172,1565-8511},
   MRCLASS = {37D35 (37A05 37A35 37B40)},
  MRNUMBER = {1942310},
MRREVIEWER = {Dieter\ H.\ Mayer},
       DOI = {10.1007/BF02785859},
       URL = {https://doi.org/10.1007/BF02785859},
}

@article {Thermo3,
    AUTHOR = {Gurevich, B. M. and Savchenko, S. V.},
     TITLE = {Thermodynamic formalism for symbolic {M}arkov chains with a
              countable number of states},
   JOURNAL = {Uspekhi Mat. Nauk},
  FJOURNAL = {Uspekhi Matematicheskikh Nauk},
    VOLUME = {53},
      YEAR = {1998},
    NUMBER = {2(320)},
     PAGES = {3--106},
      ISSN = {0042-1316,2305-2872},
   MRCLASS = {28D15 (37B10 82B20)},
  MRNUMBER = {1639451},
MRREVIEWER = {Vadim\ A.\ Ka\u{\i}manovich},
       DOI = {10.1070/rm1998v053n02ABEH000017},
       URL = {https://doi.org/10.1070/rm1998v053n02ABEH000017},
}

@article {Thermo4,
    AUTHOR = {Mauldin, R. Daniel and Urba\'{n}ski, Mariusz},
     TITLE = {Gibbs states on the symbolic space over an infinite alphabet},
   JOURNAL = {Israel J. Math.},
  FJOURNAL = {Israel Journal of Mathematics},
    VOLUME = {125},
      YEAR = {2001},
     PAGES = {93--130},
      ISSN = {0021-2172,1565-8511},
   MRCLASS = {37D35 (28D05 37A60 37B10 82B05)},
  MRNUMBER = {1853808},
MRREVIEWER = {Anthony\ Quas},
       DOI = {10.1007/BF02773377},
       URL = {https://doi.org/10.1007/BF02773377},
}

@article {Thermo5,
    AUTHOR = {Sarig, Omri M.},
     TITLE = {Thermodynamic formalism for countable {M}arkov shifts},
   JOURNAL = {Ergodic Theory Dynam. Systems},
  FJOURNAL = {Ergodic Theory and Dynamical Systems},
    VOLUME = {19},
      YEAR = {1999},
    NUMBER = {6},
     PAGES = {1565--1593},
      ISSN = {0143-3857,1469-4417},
   MRCLASS = {37A99 (37B99 37C30 37D35)},
  MRNUMBER = {1738951},
MRREVIEWER = {Irene\ Hueter},
       DOI = {10.1017/S0143385799146820},
       URL = {https://doi.org/10.1017/S0143385799146820},
}

@article {Thermo6,
    AUTHOR = {Sarig, Omri M.},
     TITLE = {Thermodynamic formalism for null recurrent potentials},
   JOURNAL = {Israel J. Math.},
  FJOURNAL = {Israel Journal of Mathematics},
    VOLUME = {121},
      YEAR = {2001},
     PAGES = {285--311},
      ISSN = {0021-2172,1565-8511},
   MRCLASS = {37D35 (37C30)},
  MRNUMBER = {1818392},
MRREVIEWER = {Pei\ Dong\ Liu},
       DOI = {10.1007/BF02802508},
       URL = {https://doi.org/10.1007/BF02802508},
}

@article {Thermo7,
    AUTHOR = {Sarig, Omri},
     TITLE = {Existence of {G}ibbs measures for countable {M}arkov shifts},
   JOURNAL = {Proc. Amer. Math. Soc.},
  FJOURNAL = {Proceedings of the American Mathematical Society},
    VOLUME = {131},
      YEAR = {2003},
    NUMBER = {6},
     PAGES = {1751--1758},
      ISSN = {0002-9939,1088-6826},
   MRCLASS = {37D35 (37A99 37B10)},
  MRNUMBER = {1955261},
MRREVIEWER = {Yves\ Coudene},
       DOI = {10.1090/S0002-9939-03-06927-2},
       URL = {https://doi.org/10.1090/S0002-9939-03-06927-2},
}

@article {Encode1,
    AUTHOR = {Buzzi, J\'{e}r\^{o}me},
     TITLE = {Intrinsic ergodicity of smooth interval maps},
   JOURNAL = {Israel J. Math.},
  FJOURNAL = {Israel Journal of Mathematics},
    VOLUME = {100},
      YEAR = {1997},
     PAGES = {125--161},
      ISSN = {0021-2172,1565-8511},
   MRCLASS = {58F11 (28D99)},
  MRNUMBER = {1469107},
MRREVIEWER = {Jan\ Kwiatkowski},
       DOI = {10.1007/BF02773637},
       URL = {https://doi.org/10.1007/BF02773637},
}

@article {Encode2,
    AUTHOR = {Hofbauer, Franz},
     TITLE = {On intrinsic ergodicity of piecewise monotonic transformations
              with positive entropy},
   JOURNAL = {Israel J. Math.},
  FJOURNAL = {Israel Journal of Mathematics},
    VOLUME = {34},
      YEAR = {1979},
    NUMBER = {3},
     PAGES = {213--237 (1980)},
      ISSN = {0021-2172},
   MRCLASS = {28D20 (58F11)},
  MRNUMBER = {570882},
MRREVIEWER = {Radu\ Nicolae\ Gologan},
       DOI = {10.1007/BF02760884},
       URL = {https://doi.org/10.1007/BF02760884},
}

@article {Encode3,
    AUTHOR = {Hofbauer, Franz},
     TITLE = {On intrinsic ergodicity of piecewise monotonic transformations
              with positive entropy. {II}},
   JOURNAL = {Israel J. Math.},
  FJOURNAL = {Israel Journal of Mathematics},
    VOLUME = {38},
      YEAR = {1981},
    NUMBER = {1-2},
     PAGES = {107--115},
      ISSN = {0021-2172},
   MRCLASS = {28D20 (58F11)},
  MRNUMBER = {599481},
MRREVIEWER = {Radu\ Nicolae\ Gologan},
       DOI = {10.1007/BF02761854},
       URL = {https://doi.org/10.1007/BF02761854},
}

@article {Encode4,
    AUTHOR = {Hofbauer, Franz},
     TITLE = {Piecewise invertible dynamical systems},
   JOURNAL = {Probab. Theory Relat. Fields},
  FJOURNAL = {Probability Theory and Related Fields},
    VOLUME = {72},
      YEAR = {1986},
    NUMBER = {3},
     PAGES = {359--386},
      ISSN = {0178-8051,1432-2064},
   MRCLASS = {58F08 (54H20 58F11)},
  MRNUMBER = {843500},
MRREVIEWER = {Takashi\ Shimano},
       DOI = {10.1007/BF00334191},
       URL = {https://doi.org/10.1007/BF00334191},
}

@article {Encode5,
    AUTHOR = {Keller, Gerhard},
     TITLE = {Lifting measures to {M}arkov extensions},
   JOURNAL = {Monatsh. Math.},
  FJOURNAL = {Monatshefte f\"{u}r Mathematik},
    VOLUME = {108},
      YEAR = {1989},
    NUMBER = {2-3},
     PAGES = {183--200},
      ISSN = {0026-9255,1436-5081},
   MRCLASS = {28D05 (54H20)},
  MRNUMBER = {1026617},
MRREVIEWER = {Vladimir\ L.\ Levin},
       DOI = {10.1007/BF01308670},
       URL = {https://doi.org/10.1007/BF01308670},
}

@article {Encode6,
    AUTHOR = {Urba\'{n}ski, Mariusz and Zdunik, Anna},
     TITLE = {Hausdorff dimension of harmonic measure for self-conformal
              sets},
   JOURNAL = {Adv. Math.},
  FJOURNAL = {Advances in Mathematics},
    VOLUME = {171},
      YEAR = {2002},
    NUMBER = {1},
     PAGES = {1--58},
      ISSN = {0001-8708,1090-2082},
   MRCLASS = {28A80 (37A05 37C45 37F10 37F50)},
  MRNUMBER = {1933383},
MRREVIEWER = {Emilia\ Petrisor},
       DOI = {10.1006/aima.2001.2067},
       URL = {https://doi.org/10.1006/aima.2001.2067},
}

@article {Encode7,
    AUTHOR = {Young, Lai-Sang},
     TITLE = {Recurrence times and rates of mixing},
   JOURNAL = {Israel J. Math.},
  FJOURNAL = {Israel Journal of Mathematics},
    VOLUME = {110},
      YEAR = {1999},
     PAGES = {153--188},
      ISSN = {0021-2172,1565-8511},
   MRCLASS = {37D25 (37A25 37D50)},
  MRNUMBER = {1750438},
MRREVIEWER = {Benoit\ Saussol},
       DOI = {10.1007/BF02808180},
       URL = {https://doi.org/10.1007/BF02808180},
}

@article {Coding1,
    AUTHOR = {Fiebig, Doris},
     TITLE = {Common extensions for locally compact {M}arkov shifts},
   JOURNAL = {Monatsh. Math.},
  FJOURNAL = {Monatshefte f\"{u}r Mathematik},
    VOLUME = {132},
      YEAR = {2001},
    NUMBER = {4},
     PAGES = {289--301},
      ISSN = {0026-9255,1436-5081},
   MRCLASS = {37B99},
  MRNUMBER = {1844068},
MRREVIEWER = {Fred\ W.\ Roush},
       DOI = {10.1007/s006050170035},
       URL = {https://doi.org/10.1007/s006050170035},
}

@article {Coding2,
    AUTHOR = {Fiebig, Doris},
     TITLE = {Graphs with pre-assigned {S}alama entropies and optimal
              degrees for locally compact {M}arkov shifts},
   JOURNAL = {Ergodic Theory Dynam. Systems},
  FJOURNAL = {Ergodic Theory and Dynamical Systems},
    VOLUME = {23},
      YEAR = {2003},
    NUMBER = {4},
     PAGES = {1093--1124},
      ISSN = {0143-3857,1469-4417},
   MRCLASS = {37B10 (05C20 37B40)},
  MRNUMBER = {1997969},
MRREVIEWER = {Fred\ W.\ Roush},
       DOI = {10.1017/S014338570200161X},
       URL = {https://doi.org/10.1017/S014338570200161X},
}

@article {Coding4,
    AUTHOR = {Fiebig, Doris and Fiebig, Ulf-Rainer},
     TITLE = {Entropy and finite generators for locally compact subshifts},
   JOURNAL = {Ergodic Theory Dynam. Systems},
  FJOURNAL = {Ergodic Theory and Dynamical Systems},
    VOLUME = {17},
      YEAR = {1997},
    NUMBER = {2},
     PAGES = {349--368},
      ISSN = {0143-3857,1469-4417},
   MRCLASS = {58F03 (28D20)},
  MRNUMBER = {1444058},
MRREVIEWER = {Fred\ W.\ Roush},
       DOI = {10.1017/S0143385797069873},
       URL = {https://doi.org/10.1017/S0143385797069873},
}

@article {Coding3,
    AUTHOR = {Fiebig, Doris},
     TITLE = {Factor theorems for locally compact {M}arkov shifts},
   JOURNAL = {Forum Math.},
  FJOURNAL = {Forum Mathematicum},
    VOLUME = {14},
      YEAR = {2002},
    NUMBER = {4},
     PAGES = {623--640},
      ISSN = {0933-7741,1435-5337},
   MRCLASS = {37B10 (37B40)},
  MRNUMBER = {1900175},
MRREVIEWER = {Kathleen\ M.\ Madden},
       DOI = {10.1515/form.2002.027},
       URL = {https://doi.org/10.1515/form.2002.027},
}

@article {Coding3.5,
    AUTHOR = {Fiebig, Doris and Roy, Mario},
     TITLE = {Factor theorems for locally compact {M}arkov shifts. {II}},
   JOURNAL = {Forum Math.},
  FJOURNAL = {Forum Mathematicum},
    VOLUME = {18},
      YEAR = {2006},
    NUMBER = {2},
     PAGES = {323--344},
      ISSN = {0933-7741,1435-5337},
   MRCLASS = {37B10 (37B40)},
  MRNUMBER = {2218424},
MRREVIEWER = {Kathleen\ M.\ Madden},
       DOI = {10.1515/FORUM.2006.019},
       URL = {https://doi.org/10.1515/FORUM.2006.019},
}

\end{document}